\documentclass[final,3p,11pt,sort&compress]{elsarticle}
\usepackage{amsmath}
\usepackage{amsfonts}
\usepackage{amssymb}
\usepackage{amsthm}
\usepackage{siunitx}
\usepackage{graphicx} % for includegraphics
\usepackage[justification=centering]{caption}
\usepackage{subcaption}
\usepackage{xspace} % needed for \eg, \ie, \etc
\usepackage{bm} % for bold math
\usepackage{threeparttable} % for tables with footnotes
\usepackage[ruled,lined,linesnumbered]{algorithm2e}
\usepackage[english]{babel}
\usepackage{amsmath}
\usepackage{booktabs}
\usepackage{placeins}
\usepackage{soul}

\usepackage{xcolor}
\usepackage{listings}

\usepackage{hyperref}
\hypersetup{
 colorlinks=true,
 citecolor=blue,
 linkcolor=blue,
 urlcolor=blue}
\usepackage{tikz}
\usetikzlibrary{matrix,backgrounds,shapes}
\usetikzlibrary{positioning}
\usetikzlibrary{fit}
\usetikzlibrary{plotmarks}
\usetikzlibrary{decorations.pathreplacing}
\usepackage{pgfplots}
\usetikzlibrary{shapes}
\usetikzlibrary{positioning,arrows}
\usetikzlibrary{fadings,shapes.arrows,shadows}
\usetikzlibrary{arrows.meta}
\definecolor{darkorange}{rgb}{1.0, 0.55, 0.0}
\definecolor{Royalblue}{rgb}{0.254,0.41,0.88}
\definecolor{royalblue}{rgb}{0.254,0.41,0.88}
\definecolor{mygreen}{rgb}{0.0, 0.7, 0.0}
\usepackage{xcolor}

\newcommand{\refone}[2][black]{\textcolor{#1}{#2}}

\newcommand{\reftwo}[2][black]{\textcolor{#1}{#2}}

\newcommand{\dadd}[1]{{\leavevmode\color{black}{#1}}}
\usepackage{gensymb}
\usepackage{tikz}
\usetikzlibrary{matrix,backgrounds,shapes}
\usetikzlibrary{positioning}
\usetikzlibrary{fit}
\usetikzlibrary{plotmarks}
\usetikzlibrary{decorations.pathreplacing}
\usepackage{pgfplots}
\usetikzlibrary{shapes}
\usetikzlibrary{positioning,arrows}
\usetikzlibrary{fadings,shapes.arrows,shadows}
\usetikzlibrary{arrows.meta}
\usepackage{bm}
\usepackage{xspace}
\usepackage{afterpage}
\usepackage{color}
\usepackage{float}
\definecolor{darkorange}{rgb}{1.0, 0.55, 0.0}
\definecolor{royalblue}{rgb}{0.254,0.41,0.88}
\usepackage{xcolor}

\newcommand{\ie}[0]{{i.e.\@}\xspace}

\DeclareMathOperator{\diag}{diag}
\newtheorem{theorem}{Theorem}[section]

\newtheorem{remark}{Remark}[section]

\newcommand{\Tr}{\ensuremath{^{\mr{\top}}}}

\newcommand{\mr}[1]{\ensuremath{\mathrm{#1}}}
\newcommand{\fnc}[1]{\ensuremath{\mathcal{#1}}}
\newcommand{\bfnc}[1]{\ensuremath{\bm{\mathcal{#1}}}}
\newcommand{\mat}[1]{\ensuremath{\mathsf{#1}}}

\newcommand{\xm}[0]{\ensuremath{x_{m}}}
\newcommand{\xone}[0]{x_{1}}
\newcommand{\xtwo}[0]{x_{2}}
\newcommand{\xthree}[0]{x_{3}}

\newcommand{\Nk}[0]{\ensuremath{N_{\kappa}}}

\newcommand{\Mk}[0]{\ensuremath{\mat{P}}_{\kappa}}

\newcommand{\Q}[0]{\ensuremath{\bm{\fnc{Q}}}}
\newcommand{\W}[0]{\ensuremath{\bm{\fnc{W}}}}
\newcommand{\Fxm}[0]{\ensuremath{\bm{\fnc{F}}_{\xm}^{(I)}}}
\newcommand{\Fxmv}[0]{\ensuremath{\bm{\fnc{F}}_{\xm}^{(V)}}}
\newcommand{\GB}[0]{\ensuremath{\bm{\fnc{G}}^{(B)}}}
\newcommand{\Gzero}[0]{\ensuremath{\bm{\fnc{G}}^{(0)}}}
\newcommand{\Uone}[0]{\ensuremath{\fnc{U}_{1}}}
\newcommand{\Utwo}[0]{\ensuremath{\fnc{U}_{2}}}
\newcommand{\Uthree}[0]{\ensuremath{\fnc{U}_{3}}}

\newcommand{\E}[0]{\ensuremath{\fnc{E}}}

\newcommand{\Cij}[2]{\ensuremath{\mat{C}_{#1,#2}}}

\newcommand{\matJk}[0]{\ensuremath{\mat{J}}_{\kappa}}

\newcommand{\onek}[0]{\ensuremath{\bm{1}}_{\Nk}}

\newcommand{\uk}[0]{\ensuremath{\bm{u}_{\kappa}}}

\newcommand{\Dxone}[0]{\ensuremath{{\mat{D}}_{x_1}}}

\newcommand{\Dxthree}[0]{\ensuremath{{\mat{D}}_{x_3}}}
\newcommand{\Dxgen}[1]{\ensuremath{{\mat{D}}_{x_{#1}}}}
\newcommand{\Poned}[0]{\ensuremath{\mat{P}_{N}}}
\newcommand{\Pmat}[0]{\ensuremath{\mat{P}}}
\newcommand{\Dmat}[0]{\ensuremath{\mat{D}}}
\newcommand{\Qmat}[0]{\ensuremath{\mat{Q}}}
\newcommand{\Bmat}[0]{\ensuremath{\mat{B}}}
\newcommand{\Doned}[0]{\ensuremath{\mat{D}_{N}}}
\newcommand{\Qoned}[0]{\ensuremath{\mat{Q}_{N}}}
\newcommand{\Boned}[0]{\ensuremath{\mat{B}_{N}}}
\newcommand{\Pxgen}[1]{\ensuremath{{\mat{P}}_{x_{#1}}}}
\newcommand{\Pxgentwod}[2]{\ensuremath{{\mat{P}}_{x_{#1},{x_#2}}}}
\newcommand{\Pxgenthreed}[3]{\ensuremath{{\mat{P}}_{x_{#1},{x_#2},{x_#3}}}}
\newcommand{\Pxgeninv}[1]{\ensuremath{{\mat{P}}^{-1}_{x_{#1}}}}

\newcommand{\Qxone}[0]{\ensuremath{{\mat{Q}}_{x_1}}}
\newcommand{\Bxone}[0]{\ensuremath{{\mat{B}}_{x_1}}}

\newcommand{\Bxtwo}[0]{\ensuremath{{\mat{B}}_{x_2}}}
\newcommand{\Qxthree}[0]{\ensuremath{{\mat{Q}}_{x_3}}}
\newcommand{\Bxthree}[0]{\ensuremath{{\mat{B}}_{x_3}}}
\newcommand{\q}[0]{\ensuremath{\bm{q}}}
\newcommand{\nxm}[0]{\ensuremath{n_{\xm}}}
\newcommand{\w}[0]{\ensuremath{\bm{w}}}

\newcommand{\Pxgentwodhat}[2]{\ensuremath{\widehat{\mat{P}}_{x_{#1},{x_#2}}}}

\newcommand{\Bxgenhat}[1]{\ensuremath{\widehat{\mat{B}}_{x_{#1}}}}

\newcommand{\XiN}[1]{\ensuremath{\bm{\Theta}_{#1}}}
\newcommand{\XiNnobold}[1]{\ensuremath{{\Theta}_{#1}}}
\newcommand{\Pmatvol}[0]{\ensuremath{\mat{P}}}

\newcommand{\IStoFxm}[1]{\ensuremath{\mat{I}_{S\text{to}F}^{x_{#1}}}}

\newcommand{\fxmIlin}[1]{\ensuremath{\bm{f}_{x_{#1}}^{(I)}}}
\newcommand{\fxmVlin}[1]{\ensuremath{\bm{f}_{x_{#1}}^{(V)}}}
\newcommand{\fxmIlinbar}[1]{\ensuremath{\bar{\bm{f}}_{x_{#1}}^{(I)}}}
\newcommand{\fxmIlinbarsc}[1]{\ensuremath{\bar{\bm{f}}_{x_{#1},sc}^{(I)}}}
\newcommand{\fxmIlinbarssr}[1]{\ensuremath{\bar{\bm{f}}_{x_{#1},ssr}^{(I)}}}

\begin{document}

\begin{frontmatter}

%%%%%%%%%%%%%%%%%%%%%%%%%%%%%%%%%%%%%%%%%%%%%%%%%%%%%%%%%%%%%%%%%%%%%%%%%%%%%%%%
\title{Conservative and entropy stable solid wall
boundary conditions for the compressible Navier--Stokes equations: Adiabatic wall and
  heat entropy transfer}

%%%%%%%%%%%%%%%%%%%%%%%%%%%%%%%%%%%%%%%%%%%%%%%%%%%%%%%%%%%%%%%%%%%%%%%%%%%%%%%%
%\tnotetext[t1]{D. Del Rey Fern\'andez was partially supported by a NSERC Postdoctoral Fellowship.}
\author[KAUST]{Lisandro Dalcin\fnref{fn1}}
%\corref{cor1} MP: LISANDRO TOLD ME THAT HE DOES NOT LIKE TO BE THE CORR. AUTHOR. I LISTED MYSELF. PLEASE, LET ME KNOW.
\ead{dalcinl@gmail.com}
\author[KAUST]{Diego Rojas\fnref{fn2}}
\ead{diego.rojasblanco@kaust.edu.sa}
\author[KAUST]{Stefano Zampini\fnref{fn1}}
\ead{stefano.zampini@kaust.edu.sa}
\author[nia,nasa]{David C.~Del Rey Fern\'andez\fnref{fn3}}
\ead{dcdelrey@gmail.com}
\author[nasa]{Mark H. Carpenter\fnref{fn4}}
\ead{mark.h.carpenter@nasa.gov}
\author[KAUST]{Matteo Parsani\fnref{fn5}\corref{cor1}}
\ead{matteo.parsani@kaust.edu.sa}

\cortext[cor1]{Corresponding author}

\fntext[fn1]{Research Scientist}
\fntext[fn2]{Ph.D. student}
\fntext[fn3]{Postdoctoral Fellow}
\fntext[fn4]{Senior Research Scientist}
\fntext[fn5]{Assistant Professor}

\address[KAUST]{King Abdullah University of Science and Technology (KAUST),
  Computer Electrical and Mathematical Science and Engineering Division (CEMSE),
  Extreme Computing Research Center (ECRC), Thuwal, Saudi Arabia}
\address[nia]{National Institute of Aerospace, Hampton, Virginia, United States}
\address[nasa]{Computational AeroSciences Branch, NASA Langley Research Center,
  Hampton, Virginia, United States}

%%%%%%%%%%%%%%%%%%%%%%%%%%%%%%%%%%%%%%%%%%%%%%%%%%%%%%%%%%%%%%%%%%%%%%%%%%%%%%%%
\begin{abstract}
We present a novel technique for the imposition of non-linear entropy conservative and entropy stable
solid wall boundary conditions for the compressible Navier--Stokes equations in the presence of an adiabatic wall,
or a wall with a prescribed heat entropy flow. The procedure relies
on the \reftwo{formalism and mimetic properties} of \dadd{diagonal-norm}, summation-by-parts and simultaneous-approximation-term operators, \reftwo{and 
is a generalization of previous works on discontinuous interface coupling \cite{parsani_entropy_stable_interfaces_2015}} and
solid wall boundary conditions \reftwo{\cite{parsani_entropy_stability_solid_wall_2015}}.
% works presented in \cite{parsani_entropy_stability_solid_wall_2015,
%parsani_entropy_stable_interfaces_2015}. \comment{MP: I am not sure if citing ourselves
%in the abstract is a good thing but this is work is actually built on top of those
%works.}

Using the method of \dadd{lines}, a semi-discrete entropy estimate for the entire
domain is obtained when the proposed numerical imposition of boundary conditions are coupled with an
\dadd{entropy-conservative} or \dadd{entropy-stable} discrete interior operator. The resulting estimate
mimics the global entropy estimate obtained at the continuous level.
The boundary data at the wall are weakly imposed using a penalty
flux approach and a simultaneous-approximation-term technique for
both the conservative variables and the gradient of the entropy variables.

Discontinuous spectral collocation operators\dadd{ (mass lumped nodal discontinuous Galerkin operators),} on high-order unstructured grids\dadd{,}
are used for the purpose of demonstrating the robustness and efficacy of the new \dadd{procedure for 
weakly enforcing boundary conditions}.
Numerical simulations confirm the non-linear stability of the proposed \dadd{technique},
with applications to three-dimensional subsonic and supersonic flows.
The procedure described is  compatible with any \dadd{diagonal-norm} summation-by-parts spatial operator, including
finite element, finite difference, finite volume, discontinuous Galerkin, and
flux reconstruction schemes.
\end{abstract}

%%%%%%%%%%%%%%%%%%%%%%%%%%%%%%%%%%%%%%%%%%%%%%%%%%%%%%%%%%%%%%%%%%%%%%%%%%%%%%%%
\begin{keyword}
  Compressible Navier--Stokes equations \sep Solid wall \sep Entropy conservation \sep Entropy stability
  \sep Summation-by-parts operators \sep Simultaneous-approximation-terms
\end{keyword}

\end{frontmatter}

%%%%%%%%%%%%%%%%%%%%%%%%%%%%%%%%%%%%%%%%%%%%%%%%%%%%%%%%%%%%%%%%%%%%%%%%%%%%%%%%
\section{Introduction}

\dadd{Next-generation} numerical algorithms for use in large eddy simulations and
direct numerical simulations of computational fluid dynamics
will rely on efficient, high-order formulations, that are able to deliver better
accuracy per degree of freedom than low-order
methods, and that feature much smaller numerical errors both in \dadd{terms} of dispersion and dissipation
\cite{hesthaven_siam_book_2017,wang_high_order_workshop_2013}.
While these properties make high-order methods well suited for time-dependent
simulations, these techniques are more prone to instability \reftwo{when} compared to their
lower\dadd{-}order \reftwo{counterparts. This is because} numerical instabilities may occur if
the flow contains discontinuities or under-resolved physical features.
Various stabilization  strategies (e.g. filtering \reftwo{\cite{hesthaven_2008_nodal_dg}}, artificial  viscosity,
over-integration, and slope limiting \reftwo{\cite{wang_high_order_workshop_2013}} to cite a few) are commonly used to address these
issues. However,\dadd{ such stabilization techniques} possess several drawbacks since i) they reduce accuracy \reftwo{\cite{wang_high_order_workshop_2013}}, ii) they usually require
\reftwo{tuning parameters} for each problem configuration, and
iii) they do not guarantee that \reftwo{solvers designed to be high-order accurate in space will not crash.}

A very promising and mathematically rigorous alternative consists in focusing
on discrete operators that are non-linearly stable\reftwo{\footnote{We use the term rigorous because, as we will see in the next sections, these
operators can mimic at the discrete level the results of the non-linear stability analysis at the continuous level.}} or, as in the case of
the compressible Navier--Stokes equations,  \emph{entropy stable}. These operators simultaneously conserve
mass, momentum\dadd{,} and total energy. \dadd{In addition}, they satisfy a discrete analogue to the conservation
or dissipation of entropy which, with positivity assumptions on temperature and density, guarantees an
$L^2$ bound on the conservative variables \cite{dafermos_book_2010,svard_weak_solutions_mod_NS_2015}. We remark that the idea of enforcing entropy stability in numerical
methods is old and commonly used for low-order operators, see e.g.
\cite{hughes_finite_element_entropy_1986,tadmor_entropy_stability_1987}. For extensions to
high-order accurate operators see
\cite{fisher_entropy_stability_fd_2013,carpenter_entropy_stability_cfd_2013,carpenter_entropy_stability_ssdc_2016,friedrich_hp_entropy_stability_2018,chan_entropy_stability_dg_2018}.

Until recently, fully discrete entropy stability was mostly established for 
implicit time stepping schemes. However, Ranocha and colleagues \cite{ranocha2019relaxation} developped and
applied new
explicit Runge--Kutta schemes (i.e., relaxation Runge--Kutta schemes)
to entropy conservative or entropy dissipative
semi-discretizations of any order for the compressible Euler and Navier--Stokes equations.
The new time integration schemes can conserve or dissipate any
  solution properties with respect to any convex functional by
  the additional of a relaxation parameter that multiplies the Runge--Kutta
  update at each step. 
The general technique is not
limited to the compressible Euler and Navier--Stokes equations setting but can be applied to many ordinary differential
equations, and to both explicit and implicit Runge--Kutta methods.

However, issues remain on the path towards complete entropy
stability \reftwo{for} the compressible Navier--Stokes equations, e.g. shock capturing
and bound-preserving limiter for high-order accurate discretizations.
One 
major obstacle is the need for boundary conditions that
preserve the entropy conservation or stability property of the interior
operator. Practical experience indicates that numerical instabilities
frequently \dadd{originate} at domain boundaries; the interaction of shocks with these
physical boundaries is particularly challenging for high\dadd{-}order formulations. An
important step towards entropy stable wall boundary conditions
for the compressible Euler and Navier--Stokes equations appears in
\cite{parsani_entropy_stability_solid_wall_2015,svard_entropy_stable_euler_wall_2014,svard_entropy_stable_solid_wall_2018}.
More specifically, non-linearly stable wall
boundary conditions for the compressible Navier--Stokes equations are presented in \cite{parsani_entropy_stability_solid_wall_2015}. Therein,
it is shown that entropy stability requires two conditions to be satisfied: i) Euler no-penetration\dadd{, and}
ii) a prescribed value for the product of
temperature and the gradient of the temperature in the normal direction to the wall.
An additional term providing a controllable numerical
dissipation has to be introduced to impose
a zero relative velocity at the wall, i.e. the no-slip condition. Therefore,
the solid wall boundary conditions proposed in \cite{parsani_entropy_stability_solid_wall_2015}
are entropy stable, but not entropy conservative. Note that in \cite{svard_entropy_stable_solid_wall_2018}
it is shown that demanding a bound on velocity gradients necessitates the use
of the full no-slip conditions, i.e. the thermal and the relative velocity boundary
conditions.

In this work we present a general procedure for the development of point-wise entropy conservative
boundary conditions representing either an
adiabatic solid wall or a wall with a prescribed heat entropy flow
for the compressible Navier--Stokes equations, discretized by using \dadd{diagonal-norm}, summation-by-parts (SBP) and simultaneous-approximation-term (SAT) operators
(i.e. SBP-SAT operators).
Entropy conservation is obtained by penalizing, using a SAT penalty, both the entropy variables
and their gradients in the normal direction to the wall\dadd{,} \dadd{as in} the local discontinuous
Galerkin approach \cite{doi:10.1137/S0036142997316712}. The overall algorithm closely follows the
treatment of the discontinuous interior interfaces coupling presented in
\cite{parsani_entropy_stable_interfaces_2015}; a single implementation, with different inputs,
can be used for interface penalization and imposition of boundary conditions.
A controllable \reftwo{amount of} dissipation can be added to make the boundary conditions entropy stable.
The new procedure can be immediately applied to a moving wall, as will be shown in
the theoretical and numerical results sections. 

%The penalization of the gradients and the possibility to construct entropy conservative solid wall boundary conditions
%represent the main novelty of this
%work within the context SBP-SAT discretizations for the imposition of boundary
%conditions \cite{parsani_entropy_stability_solid_wall_2015}.
%At the semi-discrete
%level, the proposed techniques exactly mimic the boundary
%contribution obtained by applying the entropy stability analysis to the
%continuous, compressible Navier--Stokes equations.

The manuscript is organized as follows. A brief review concerning the derivation of
continuous entropy inequalities and the entropy analysis of the viscous wall
boundary conditions for the compressible Navier--Stokes equations is provided in Section \ref{sec:entropy_continuous}.
%In  Section  \ref{sec:entropy_stable_discretization_gll},  we  review  how  to
%construct  the interior entropy  stable  SBP-SAT operators on affine tensor
%product elements.
\dadd{The weak,} point-wise\dadd{, imposition of} entropy conservative and entropy stable boundary conditions is carried out in Section \ref{sec:sc_adiabatic_no_slip_bc} for an adiabatic solid wall and for a wall with a prescribed heat entropy transfer.
Section \ref{sec:numerical_results} presents numerical results which
confirm the accuracy and stability of the proposed boundary conditions.
\reftwo{Conclusions} are drawn in Section \ref{sec:conclusions}. \dadd{Finally, in \ref{app:python_proofs} a Python script is provided that symbolically verifies all proofs for curvilinear grids, 
while in  \ref{app:fortran_code} a simple and dimension-agnostic implementation of the entropy stable solid wall boundary 
condition coded in FORTRAN is presented.}

%%%%%%%%%%%%%%%%%%%%%%%%%%%%%%%%%%%%%%%%%%%%%%%%%%%%%%%%%%%%%%%%%%%%%%%%%%%%%%%%
\section{A brief review of entropy stability theory}\label{sec:entropy_continuous}
%for the compressible
%Navier--Stokes equations and the solid wall boundary conditions
In this Section\dadd{,} we review the continuous entropy theory
for the compressible Navier--Stokes equations and the solid wall boundary
conditions by closely following \cite{carpenter_entropy_stable_staggered_2015,parsani_entropy_stability_solid_wall_2015}.

%%%%%%%%%%%%%%%%%%%%%%%%%%%%%%%%%%%%%%%%%%%%%%%%%%%%%%%%%%%%%%%%%%%%%%%%%%%%%%%%
\subsection{The compressible Navier--Stokes equations}\label{subsec:compressible_ns}
To keep the presentation simple but without loss of generality, we consider the three-dimensional compressible Navier--Stokes equations in Cartesian coordinates $\left(\xone,\xtwo,\xthree\right)$ for an ideal gas in a bounded domain $\Omega$ with boundary $\Gamma$
\begin{equation}\label{eq:compressible_ns}
\begin{split}
&\frac{\partial\Q}{\partial t}+\sum\limits_{m=1}^{3}\frac{\partial \Fxm}{\partial \xm} =\sum\limits_{m=1}^{3}\frac{\partial\Fxmv}{\partial\xm}, \quad
\forall \left(\xone,\xtwo,\xthree\right)\in\Omega,\quad t\ge 0,\\
&\Q\left(\xone,\xtwo,\xthree,t\right)=\GB\left(\xone,\xtwo,\xthree,t\right),\quad\forall \left(\xone,\xtwo,\xthree\right)\in\Gamma,\quad t\ge 0,\\
&\Q\left(\xone,\xtwo,\xthree,0\right)=\Gzero\left(\xone,\xtwo,\xthree,0\right), \quad
\forall \left(\xone,\xtwo,\xthree\right)\in\Omega.
\end{split}
\end{equation}
The vectors $\Q$, $\Fxm$\dadd{,} and $\Fxmv$ respectively denote the conserved variables, the inviscid ($I$) fluxes, and the viscous ($V$) fluxes. The boundary data\dadd{,}
$\GB$, and the initial condition\dadd{,} $\Gzero$, are assumed to be in $L^{2}(\Omega)$, with the further assumption that $\GB$ will be set to coincide with linear well posed boundary conditions and such that entropy conservation/stability is achieved.

The vector of conserved variables is given as
\begin{equation}\label{eq:conserved_vars}
\Q = \left[\rho,\rho\Uone,\rho\Utwo,\rho\Uthree,\rho\E\right]\Tr,
\end{equation}
where $\rho$ denotes the density, $\bm{\fnc{U}} = \left[\Uone,\Utwo,\Uthree\right]\Tr$ is the velocity
vector, and $\E$ is the specific total energy. The inviscid fluxes are given by
\begin{equation}\label{eq:inviscid_flux}
\Fxm = \left[\rho\fnc{U}_{m},\rho\fnc{U}_{m}\Uone+\delta_{m,1}\fnc{P},\rho\fnc{U}_{m}\Utwo+\delta_{m,2}\fnc{P},\rho\fnc{U}_{m}\Uthree+\delta_{m,3}\fnc{P},\rho\fnc{U}_{m}\fnc{H}\right]\Tr,
\end{equation}
where $\fnc{P}$ is the pressure, $\fnc{H}$ is the specific total enthalpy\dadd{,} and $\delta_{i,j}$ is the
Kronecker delta. The viscous flux $\Fxmv$ is given as
\begin{equation}\label{eq:Fv}
\Fxmv=\left[0,\tau_{1,m},\tau_{2,m},\tau_{3,m},\sum\limits_{i=1}^{3}\tau_{i,m}\fnc{U}_{i}-\kappa\frac{\partial \fnc{T}}{\partial\xm}\right]\Tr,
\end{equation}
\dadd{where} $\kappa = \kappa(T)$ \dadd{is} thermal conductivity, and the viscous stresses \dadd{is }given by
\begin{equation}\label{eq:tau}
\tau_{i,j} = \mu\left(\frac{\partial\fnc{U}_{i}}{\partial x_{j}}+\frac{\partial\fnc{U}_{j}}{\partial x_{i}}
-\delta_{i,j}\frac{2}{3}\sum\limits_{n=1}^{3}\frac{\partial\fnc{U}_{n}}{\partial x_{n}}\right),
\end{equation}
\dadd{where} $\mu = \mu(T)$ \dadd{is }the dynamic viscosity.

The required constitutive relations are
\begin{equation*}
\fnc{H} = c_{\fnc{P}}\fnc{T}+\frac{1}{2}\bm{\fnc{U}}\Tr\bm{\fnc{U}},\quad \fnc{P} = \rho R \fnc{T},\quad R = \frac{R_{u}}{M_{w}},
\end{equation*}
where $c_{\fnc{P}}$ is the specific heat at constant pressure, \fnc{T} is the temperature, $R_{u}$ is the universal gas constant, and $M_{w}$ is the molecular weight of the gas. Finally,
the thermodynamic entropy is given as
\begin{equation*}
s=\frac{R}{\gamma-1}\log\left(\frac{\fnc{T}}{\fnc{T}_{\infty}}\right)-R\log\left(\frac{\rho}{\rho_{\infty}}\right),\quad \gamma=\frac{c_{\fnc{P}}}{c_{\fnc{P}}-R},
\end{equation*}
with $\fnc{T}_{\infty}$ and $\rho_{\infty}$ the reference temperature and density, respectively.

It is well known that the compressible Navier--Stokes equations given in \eqref{eq:compressible_ns} possess
a convex extension that, when integrated over the physical domain $\Omega$, only depends on the boundary \reftwo{data.
Such an extension} yields the entropy function
\begin{equation}\label{eq:entropy_function}
\fnc{S}=-\rho s,
\end{equation}
which is a useful tool for proving
stability in the $L^{2}$ norm \cite{dafermos_book_2010,svard_weak_solutions_mod_NS_2015}. We can then define the entropy
variables $\bm{W}=\partial\fnc{S}/\partial\Q$, \reftwo{which} for the compressible Navier--Stokes equations given in \eqref{eq:compressible_ns} \reftwo{are}
\begin{equation}\label{eq:entropy_variables}
  \W = \left[\frac{\fnc{H}-1/2\,\dadd{\bfnc{U}\Tr\bfnc{U}}}{\fnc{T}}-s-\frac{\dadd{\bfnc{U}\Tr\bfnc{U}}}{\fnc{T}},\frac{\Uone}{\fnc{T}},\frac{\Utwo}{\fnc{T}},\frac{\Uthree}{\fnc{T}},-\frac{1}{\fnc{T}}\right]\Tr.
\end{equation}
We remark that the convexity of $\fnc{S}$ guarantees the invertibility of the
mapping between conservative and entropy variables, provided that
the temperature $\fnc{T}$ and the density $\rho$ are positive. In what follows, we always assume that such positivity is preserved. %; positivity preservation is outside the scope of this study.

The \dadd{vector of} entropy variables simultaneously
contracts all of the inviscid spatial fluxes $\Fxm$  as
\begin{equation}\label{eq:contraction_inviscid_flux}
\frac{\partial \fnc{S}}{\partial \bfnc{Q}}\frac{\partial\Fxm}{\partial \xm}=
\frac{\partial \fnc{S}}{\partial \bfnc{Q}}\frac{\partial\Fxm}{\partial \bfnc{Q}}\frac{\partial\bfnc{Q}}{\partial \xm}=
\frac{\partial\fnc{F}_{\xm}}{\partial\bfnc{Q}}\frac{\partial\bfnc{Q}}{\partial \xm}=
\frac{\partial\fnc{F}_{\xm}}{\partial \xm},\qquad m=1,2,3,
\end{equation}
where the
scalar $\fnc{F}_{\xm}(\bfnc{Q})$ denotes the entropy flux in the $m$-th direction.
By letting $\bfnc{W}$ \reftwo{take} the role \reftwo{as of} a new set of independent variables,
\ie $\bfnc{Q}=\bfnc{Q}(\bfnc{W})$, the entropy variables \eqref{eq:entropy_variables} symmetrize
the system \eqref{eq:compressible_ns} as \reftwo{\cite{dutt_1988_stble_bc_ns}}
\begin{equation}\label{eq:symmetric_nse}
%\frac{\partial\bfnc{Q}}{\partial t}+\sum\limits_{m=1}^{3}\frac{\partial\Fxm}{\partial \xm} = \sum\limits_{m=1}^{3}\frac{\partial\Fxmv}{\partial\xm}=
  \frac{\partial\bfnc{Q}}{\partial\bfnc{W}}\frac{\partial\bfnc{W}}{\partial t}+\sum\limits_{m=1}^{3}\frac{\partial\bfnc{F}_{\xm}^{(I)}}{\partial\bfnc{W}}\frac{\partial\bfnc{W}}{\partial \xm} = \sum\limits_{m,j=1}^{3}\frac{\partial}{\partial\xm} \left(\Cij{m}{j}\frac{\partial\W}{\partial x_j}\right), %=0,
\end{equation}
where the viscous fluxes $\Fxmv$ have been recast in term of the entropy variables as
\begin{equation}\label{eq:Fxment}
\Fxmv=\sum\limits_{j=1}^{3}\Cij{m}{j}\frac{\partial\bfnc{W}}{\partial x_{j}}.
\end{equation}
For the definition of the symmetric \dadd{and} semi-definite $\Cij{m}{j}$ matrices see \cite{fisher_phd_2012,parsani_entropy_stability_solid_wall_2015}.

Due to the \reftwo{symmetric nature} of $\partial\bfnc{Q}/\partial\bfnc{W}$ and $\partial\bfnc{F}_{\xm}^{(I)}/\partial\bfnc{W}$,
there exist scalar functions in entropy variables whose Jacobians represent the conservative variables $\Q$ and the inviscid fluxes $\Fxm$ as
\begin{equation}
  \bfnc{Q}\Tr=\frac{\partial\fnc{\varPhi}}{\partial\bfnc{W}},\quad\left(\bfnc{F}_{\xm}^{(I)}\right)^{\mr{\top}}=\frac{\partial\fnc{\varPsi}_{\xm}}{\partial\bfnc{W}}.
\end{equation}
$\fnc{\varPhi}$ is the \emph{potential}, whereas the $\fnc{\varPsi}_{\xm}$ functions are the
\emph{potential fluxes}\dadd{ in the $\xm$ direction}, with $\left(\fnc{\varPhi},\fnc{\varPsi}_{\xm}\right)$ the \emph{potential-potential flux} pair \cite{tadmor_review_entropy_analysis_2003}.
A close relation between entropy and the potential-potential flux pair is summarized in the following Theorem \cite{Godunov1961}, which is due to Godunov (see also \cite{harten1983}).
\begin{theorem}
\label{th:Godunov}
If a system of conservation laws can be symmetrized by introducing new variables $\bfnc{W}$, and
$\bfnc{Q}$ is a convex function of $\fnc{\varPhi}$, then an entropy function
$\fnc{S}=\fnc{S}(\bfnc{Q})$ is given by
\begin{equation}
\label{eq:GodSpotential}
 \fnc{\varPhi} = \bfnc{W}\Tr\bfnc{Q} - \fnc{S},
\end{equation}
and the entropy fluxes $\fnc{F}_{\xm}(\bfnc{Q})$ satisfy
\begin{equation}
\label{eq:GodFpotential}
\quad \fnc{\varPsi}_{\xm} = \bfnc{W}\Tr \Fxm - \fnc{F}_{\xm}.
\end{equation}
\end{theorem}
%\begin{proof}
%The proof of this theorem can be found in references~\cite{Godunov1961,harten1983}.
%\end{proof}

\reftwo{By contracting the system of equations \eqref{eq:compressible_ns} with the entropy variables,
\begin{equation}
\label{eq:entropy_equation_rev2}
\begin{aligned}
  \frac{\partial \fnc{S}}{\partial \Q} \frac{\partial \Q}{\partial t} +
  \sum\limits_{m=1}^{3}\frac{\partial \fnc{S}}{\partial \Q} \frac{\partial \Fxm}{\partial \xm} & =
  \sum\limits_{m=1}^{3}\frac{\partial \fnc{S}}{\partial \Q} \frac{\partial \Fxmv}{\partial \xm},
\end{aligned}
\end{equation}
and applying the relations given in \eqref{eq:contraction_inviscid_flux},
\eqref{eq:symmetric_nse}, and \eqref{eq:Fxment}, we arrive at the differential form of the (scalar) entropy equation
\begin{equation}
\label{eq:entropy_equation}
\begin{aligned}
  \sum\limits_{m=1}^{3}\left[\frac{\partial \fnc{S}}{\partial t} + \frac{\partial \fnc{F}_{\xm}}{\partial \xm}\right]
  & =\sum\limits_{m=1}^{3}\left(\frac{\partial}{\partial \xm}\left(\W^{\top} \Fxmv\right) -
  \left(\frac{\partial \W}{\partial \xm}\right)^{\top} \Fxmv\right) \\
  & =\sum\limits_{m=1}^{3}\left(\frac{\partial}{\partial \xm}\left(\W^{\top} \Fxmv\right) -
  \sum\limits_{j=1}^{3}\left(\frac{\partial \W}{\partial x_m}\right)^{\top} \Cij{m}{j} \,
  \frac{\partial \W}{\partial x_j}\right).
\end{aligned}
\end{equation}
}

To obtain a global conservation statement for the entropy function $\fnc{S}$,
we then integrate equation \eqref{eq:entropy_equation} over the domain $\Omega$
%\comment{accounting for local sources of
%entropy dissipation}
%and obtain a global conservation statement for the entropy
%\comment{MP: equation \eqref{eq:entropy_equation} is a PDE (strong form), thus, if
%one wants to get a balance for the entropy function in the domain $\Omega$ that
%equation has to be integrated.}
%\comment{(review fixes)}
\begin{equation}
 \label{eq:continuous_entropy_estimate_discont}
  \begin{split}
  \frac{\mr{d}}{\mr{d}t}\int_{\Omega}\fnc{S}\mr{d}\Omega
  &\le \sum\limits_{m=1}^{3}\int_{\Gamma}\left(\W\Tr \Fxmv - \fnc{F}_{\xm} \right)\nxm\mr{d}\Gamma
    -\sum_{m,j=1}^3 \int_{\Omega} \left(\frac{\partial \W}{\partial x_m}\right)^{\top} \Cij{m}{j} \,
  \frac{\partial \W}{\partial x_j} \mr{d}\Omega\\
&=\sum\limits_{m=1}^{3}\int_{\Gamma}\left(\W\Tr \Fxmv - \fnc{F}_{\xm} \right)\nxm\mr{d}\Gamma - DT,
\end{split}
\end{equation}
where $\nxm$ is the $m$-th component of the outward facing unit normal and
\[
DT = \sum_{m,j=1}^3 \int_{\Omega} \left(\frac{\partial \W}{\partial x_m}\right)^{\top} \Cij{m}{j} \,
  \frac{\partial \W}{\partial x_j} \mr{d}\Omega.
\]

We remark that viscous dissipation always introduces a negative rate of change in entropy, 
\dadd{since} the $-DT$ term in \eqref{eq:continuous_entropy_estimate_discont} is negative semi-definite. 
An increase in entropy within the domain can only result from
data that convects or diffuses through the boundaries $\Gamma$. For smooth flows, we finally note that the inequality sign in \eqref{eq:continuous_entropy_estimate_discont} becomes an equality.

%Physical arguments require that the five-by-five matrices $\cijh$ be
%positive semi-definite \cite{Fisher2012dissertation,FisherCarpenter2013JCP}.

%%%%%%%%%%%%%%%%%%%%%%%%%%%%%%%%%%%%%%%%%%%%%%%%%%%%%%%%%%%%%%%%%%%%%%%%%%%%%%%%
\subsection{No-slip wall boundary conditions}\label{subsec:no_slip_bc}

For simplicity, we let the domain of interest be $\Omega=[0,1]^3$ and we only consider entropy
conservation (i.e., the equality relation in \eqref{eq:continuous_entropy_estimate_discont}). 
Thus, expanding the notation in equation
\eqref{eq:continuous_entropy_estimate_discont} yields
\begin{equation}\label{eq:continuous_entropy_estimate}
\begin{aligned}
  & \frac{\mr{d}}{\mr{d}t} \int_{\Omega} \fnc{S} \, \mr{d}x_1 \, \mr{d}x_2 \, \mr{d}x_3 =  -DT  \\
  & \:+\: \int_{x_1=0}  \left[+\fnc{F}_{x_1} - \, \W^{\top}
  \left(\Cij{1}{1} \frac{\partial \W}{\partial x_1}
  +     \Cij{1}{2} \frac{\partial \W}{\partial x_2} +
        \Cij{1}{3} \frac{\partial \W}{\partial x_3} \right)
                                                         \right] \mr{d}x_2 \, \mr{d}x_3 \\
  & \:+\: \int_{x_1=1}  \left[-\fnc{F}_{x_1} +  \, \W^{\top}
  \left(\Cij{1}{1} \frac{\partial \W}{\partial x_1}
  +     \Cij{1}{2} \frac{\partial \W}{\partial x_2} +
        \Cij{1}{3} \frac{\partial \W}{\partial x_3} \right)
                                                         \right] \mr{d}x_2 \, \mr{d}x_3 \\
  & \:+\: \int_{x_2=0}  \left[+\fnc{F}_{x_2} - \, \W^{\top}
  \left(\Cij{2}{1} \frac{\partial \W}{\partial x_1}
  +     \Cij{2}{2} \frac{\partial \W}{\partial x_2} +
        \Cij{2}{3} \frac{\partial \W}{\partial x_3} \right)
                                                         \right] \mr{d}x_1 \, \mr{d}x_3 \\
  & \:+\: \int_{x_2=1} \left[-\fnc{F}_{x_2} + \, \W^{\top}
  \left(\Cij{2}{1} \frac{\partial \W}{\partial x_1}
  +     \Cij{2}{2} \frac{\partial \W}{\partial x_2} +
        \Cij{2}{3} \frac{\partial \W}{\partial x_3} \right)
                                                         \right] \mr{d}x_1 \, \mr{d}x_3 \\
  & \:+\: \int_{x_3=0}  \left[+\fnc{F}_{x_3} - \, \W^{\top}
  \left(\Cij{3}{1} \frac{\partial \W}{\partial x_1}
  +     \Cij{3}{2} \frac{\partial \W}{\partial x_2} +
        \Cij{3}{3} \frac{\partial \W}{\partial x_3} \right)
                                                         \right] \mr{d}x_1 \, \mr{d}x_2 \\
  & \:+\: \int_{x_3=1} \left[-\fnc{F}_{x_3} + \, \W^{\top}
  \left(\Cij{3}{1} \frac{\partial \W}{\partial x_1}
  +     \Cij{3}{2} \frac{\partial \W}{\partial x_2} +
        \Cij{3}{3} \frac{\partial \W}{\partial x_3} \right)
                                                         \right] \mr{d}x_1 \, \mr{d}x_2 \, .
\end{aligned}
\end{equation}
Note that the plus and minus signs within the integrand terms of \eqref{eq:continuous_entropy_estimate} account for the
direction of the outward facing normals \dadd{on} the six faces of the unit cube $\Omega$.

Furthermore, without loss of generality, we consider the case of a wall placed at $x_1=0$ \reftwo{such that the normal vector is \dadd{$\bm n = (-1,0,0)^{\top}$}}, and we assume that all the other boundaries terms are
entropy conservative, which allows us to neglect their contributions. Then, estimate \eqref{eq:continuous_entropy_estimate} reduces to
\begin{equation}\label{eq:continuous_entropy_estimate_1wall}
\begin{aligned}
  & \frac{\mr{d}}{\mr{d}t} \int_{\Omega} S \, \mr{d}x_1 \, \mr{d}x_2 \, \mr{d}x_3 =  - DT  \\
  & \:+\: \int_{x_1=0}  \left[\fnc{F}_{x_1} - \, \W^{\top}
  \left(\Cij{1}{1} \frac{\partial \W}{\partial x_1}
  +     \Cij{1}{2} \frac{\partial \W}{\partial x_2} +
        \Cij{1}{3} \frac{\partial \W}{\partial x_3} \right)
                                                         \right] \mr{d}x_2 \, \mr{d}x_3 \,.
\end{aligned}
\end{equation}
%To bound the time derivative of the integral of the entropy function $\fnc{S}$, the right-hand side of equation
%\eqref{eq:continuous_entropy_estimate_1wall} requires boundary data.

%As summarized in \cite{parsani_entropy_stability_solid_wall_2015} and in the references
%cited therein, for a solid viscous wall, assuming linear analysis,
%four independent boundary conditions must be imposed
% \cite{kreiss_ibvc_book_1989}.\footnote{Using the linear analysis, it can be shown
% that a
% solid viscous wall
% behaves like a subsonic outflow
%\cite{nordstrom_well_posedness_2005}.} Three boundary conditions are the
%no-slip boundary conditions, $\Uone = \Utwo = \Uthree = 0$ (i.e., the velocity vector
%relative to the wall is the zero vector). In the linear settings (see, for instance,
%\cite{svard_linear_wall_bc_2008,berg_linear_robin_wall_bc_2011}), the fourth condition is the
%gradient of the temperature normal to the wall,
%$(\partial \fnc{T} / \partial n)_{wall}$, (Neumann boundary condition; e.g., the adiabatic wall),
%or the temperature of the wall, $\fnc{T}_{wall}$, (the Dirichlet or isothermal wall boundary condition),
%or a mixture of Dirichlet and Neumann conditions (the Robin boundary condition).
%%The first case with $\partial T_w / \partial n = 0$ represents the adiabatic wall
%%condition whereas the second one is the so-called isothermal wall condition.
%These four boundary conditions lead to energy stability
%(linear stability); see, for instance, \cite{svard_linear_wall_bc_2008,berg_linear_robin_wall_bc_2011}.
Within the context of linear analysis for a solid viscous wall, the wall behaves like a subsonic outflow
\cite{nordstrom_well_posedness_2005}, \dadd{and} four independent boundary conditions must be imposed to prove energy (linear) stability
 \cite{kreiss_ibvc_book_1989,svard_linear_wall_bc_2008,berg_linear_robin_wall_bc_2011} \dadd{(}see also \cite{parsani_entropy_stability_solid_wall_2015} and the references therein\dadd{)}. The first three correspond to the
no-slip boundary conditions $\Uone = \Utwo = \Uthree = 0$ that impose a zero relative velocity with respect to the wall.
The fourth condition can be either imposed on the
gradient of the temperature normal to the wall
$(\partial \fnc{T} / \partial n)_{wall}$ (Neumann boundary condition\dadd{,} e.g. the adiabatic wall),
or to the temperature at the wall $\fnc{T}_{wall}$, (the Dirichlet or isothermal wall boundary condition),
or a mixture of these two (the Robin boundary condition) \cite{svard_linear_wall_bc_2008,berg_linear_robin_wall_bc_2011}.
%The first case with $\partial T_w / \partial n = 0$ represents the adiabatic wall
%condition whereas the second one is the so-called isothermal wall condition.

%In the remainder of this section, we will show a new formulation of the adiabatic wall boundary
%conditions that have to be imposed to bound estimate \eqref{eq:continuous_entropy_estimate_1wall} and, hence, to attain entropy conservation
%and entropy stability.
%We re-iterate here Theorems 3.1 and 3.2 reported in \cite{parsani_entropy_stability_solid_wall_2015}.
%These theorems provide the conditions to bound the time rate of change of the entropy function

In the non-linear case, entropy conservation and entropy stability in the adiabatic solid wall case or
a wall with a prescribed heat entropy flow are attained by means of the next two theorems.
These theorems provide the conditions \dadd{that result in a bound on} the time rate of change of the entropy function
in \eqref{eq:continuous_entropy_estimate_1wall}, and are point-wise valid \cite{parsani_entropy_stability_solid_wall_2015}.
The first theorem is a generalization of Theorem 3.1 presented
in \cite{parsani_entropy_stability_solid_wall_2015} to a moving wall.
\begin{theorem}\label{th:continous_inviscid_bound}
  The no-slip boundary conditions $\ensuremath{\fnc{U}_{1}} = 0$ and
  $\ensuremath{\fnc{U}_{m}} = \ensuremath{\fnc{U}_{m}^{wall}}, \, m=2,3$ bound the inviscid
contribution to the time derivative of the entropy in equation
\eqref{eq:continuous_entropy_estimate_1wall}.
\end{theorem}
\begin{proof}
Equation \eqref{eq:GodFpotential} provides the following relation for the entropy flux
\begin{equation}
	\fnc{F}_{x_1} = \W^{\top} \ensuremath{\bm{\fnc{F}}_{x_1}^{(I)}} - {\fnc{\varPsi}}_{x_1}=-\rho s\Uone R, \quad {\fnc{\varPsi}}_{x_1} = -\rho \Uone R.
\end{equation}
Substituting the no-slip conditions,  into the definition
of the inviscid flux, $\ensuremath{\bm{\fnc{F}}_{x_1}^{(I)}}$, (equation \eqref{eq:inviscid_flux})
and the condition $\Uone = 0$ into the definition
of $\dadd{{\fnc{\varPsi}}_{x_1}}$, yields the desired result  $\fnc{F}_{x_1} = 0$.
\end{proof}
\begin{remark}
  In a general setting, the no-slip boundary conditions read as
  $\ensuremath{\bm{\fnc{U}}} = \ensuremath{\bm{\fnc{U}}}^{wall}$ with
  $\ensuremath{\bm{\fnc{U}}}^{wall}~\cdot~{\bm n } = 0$, where $\ensuremath{\bm{\fnc{U}}}$, $\ensuremath{\bm{\fnc{U}}}^{wall}$
  and ${\bm{n}}$ denote the velocity vector of the fluid, the velocity vector of the wall and the unit normal vector,
  respectively.
\end{remark}

%\begin{theorem}\label{th:continous_inviscid_bound}
%The no-slip boundary conditions $\Uone = \Utwo = \Uthree = 0$ bound the inviscid
%contribution to the time derivative of the entropy in equation
%\eqref{eq:continuous_entropy_estimate_1wall}.
%\end{theorem}
%\begin{proof}
%  See Theorem 3.1 in \cite{parsani_entropy_stability_solid_wall_2015}.
%\end{proof}
\begin{theorem}\label{th:continous_viscous_bound}
  \reftwo{The boundary condition
  \begin{equation}\label{eq:temperature_continous}
    \mathtt{g}(t) = \kappa \frac{\partial \fnc{T}}{\partial n} \frac{1}{\fnc{T}},
  \end{equation}
	where $\partial \fnc{T}/\partial n$ denotes the normal derivative of $\fnc{T}$, bounds the viscous contribution to the time derivative of the entropy
	\eqref{eq:continuous_entropy_estimate_1wall}}.
\end{theorem}
\begin{proof}
  See Theorem 3.2 in \cite{parsani_entropy_stability_solid_wall_2015}.
\end{proof}

\begin{remark}
	\reftwo{The scalar value, 
\begin{equation*}
  \kappa\left(\frac{\partial \fnc{T}}{\partial x_1} \frac{1}{\fnc{T}}\right) = \W^{\top}
  \left(\Cij{1}{1} \frac{\partial \W}{\partial x_1}
  +     \Cij{1}{2} \frac{\partial \W}{\partial x_2} +
	\Cij{1}{3} \frac{\partial \W}{\partial x_3} \right),
\end{equation*}}
  accounts for the change in entropy due to the wall heat flux at $x_1 = 0$ \cite{parsani_entropy_stability_solid_wall_2015} and
  is often denoted as heat entropy transfer or heat entropy flow \cite{bejan_entropy_generation_book}.
\end{remark}

%%%%%%%%%%%%%%%%%%%%%%%%%%%%%%%%%%%%%%%%%%%%%%%%%%%%%%%%%%%%%%%%%%%%%%%%%%%%%%%%
\section{Entropy conservative and entropy stable solid wall boundary conditions
for the semi-discrete system}\label{sec:sc_adiabatic_no_slip_bc}

%When solving system \eqref{eq:compressible_ns} numerically,
\dadd{To discretize in space}, we partition the physical domain $\Omega$ into non-overlapping hexahedral elements and we semi-discretize the system \eqref{eq:compressible_ns} using a multi-dimensional SBP operator, constructed from a one-dimensional SBP operator by \reftwo{way of} tensor
products. The nodal distribution within each element is based on
$N^3$ Legendre-Gauss-Lobatto (LGL) points \cite{carpenter_ssdc_2014,parsani_entropy_stability_solid_wall_2015,carpenter_entropy_stability_ssdc_2016,fernandez_entropy_stable_p_refinement_2018}, where $N$ is the number of LGL point in one direction.

Here, we summarize the relevant SBP operators used to discretize \eqref{eq:compressible_ns},
and to derive the new procedure to impose the solid wall boundary conditions.
\begin{equation}\label{eq:SBP-tensor-matrices}
\begin{gathered}
  \Dxone = \left(\Doned \otimes \dadd{\mat{I}_{N}} \otimes \dadd{\mat{I}_{N}} \otimes \dadd{\mat{I}_{5}}\right), \quad \cdots \quad
% \dmat_{x_2} = \left(I_{N} \otimes \dmat_{N} \otimes I_{N} \otimes I_{5}\right), \\
  \Dxthree = \left(\dadd{\mat{I}_{N}} \otimes \dadd{\mat{I}_{N}} \otimes \Doned \otimes \dadd{\mat{I}_{5}}\right), \\ \\
  \Qxone = \left(\Qoned \otimes \dadd{\mat{I}_{N}} \otimes \dadd{\mat{I}_{N}} \otimes \dadd{\mat{I}_{5}}\right), \quad  \cdots \quad
% \bmat_{x_2} = \left(I_{N} \otimes \bmat_{N} \otimes I_{N} \otimes I_{5}\right), \\
  \Qxthree = \left(\dadd{\mat{I}_{N}} \otimes \dadd{\mat{I}_{N}} \otimes \Qoned \otimes \dadd{\mat{I}_{5}}\right), \\ \\
  \Bxone = \left(\Boned \otimes \dadd{\mat{I}_{N}} \otimes \dadd{\mat{I}_{N}} \otimes \dadd{\mat{I}_{5}}\right), \quad  \cdots \quad
  \Bxthree = \left(\dadd{\mat{I}_{N}} \otimes \dadd{\mat{I}_{N}} \otimes \Boned \otimes \dadd{\mat{I}_{5}}\right), \\ \\
  {\Delta}_{x_1} = \left(\Delta_{N} \otimes \dadd{\mat{I}_{N}} \otimes \dadd{\mat{I}_{N}} \otimes \dadd{\mat{I}_{5}}\right), \quad\cdots \quad
% \Delta_{x_2} = \left(I_{N} \otimes \Delta_{N} \otimes I_{N} \otimes I_{5}\right), \\
  {\Delta}_{x_3} = \left(\dadd{\mat{I}_{N}} \otimes \dadd{\mat{I}_{N}} \otimes \Delta_{N} \otimes \dadd{\mat{I}_{5}}\right)\dadd{,} \\ \\
  \Pxgen{1} = \left(\Poned \otimes \dadd{\mat{I}_{N}} \otimes \dadd{\mat{I}_{N}} \otimes \dadd{\mat{I}_{5}}\right), \quad \cdots \quad
% \pmat_{x_2} = \left(I_{N} \otimes \pmat_{N} \otimes I_{N} \otimes I_{5}\right), \\
  \Pxgen{3} = \left(\dadd{\mat{I}_{N}} \otimes \dadd{\mat{I}_{N}} \otimes \Poned \otimes \dadd{\mat{I}_{5}}\right), \quad  \\ \\
  \Pxgentwod{1}{2} = \left(\Poned \otimes \Poned \otimes \dadd{\mat{I}_{N}} \otimes \dadd{\mat{I}_{5}}\right), \quad \cdots \quad
% \pmat_{x_1 x_3} = \left(\pmat_{N} \otimes I_{N} \otimes \pmat_{N} \otimes I_{5}\right), \quad
  \Pxgentwod{2}{3} = \left(\dadd{\mat{I}_{N}} \otimes \Poned \otimes \Poned \otimes \dadd{\mat{I}_{5}}\right), \\ \\
  \Pmatvol = \Pxgenthreed{1}{2}{3} = \left(\Poned \otimes \Poned \otimes \Poned \otimes \dadd{\mat{I}_{5}} \right), \\ \\
 %\cmat_{ij}  = \left(I_{N} \otimes I_{N} \otimes I_{N} \otimes \widehat{c}_{ij} \right),
\end{gathered}
\end{equation}
$\Doned$, $\Qoned$, $\Boned$, $\Delta_N$ and $\Poned$  are the
one-dimensional SBP operators, and  $\dadd{\mat{I}_{N}}$ is
the identity
matrix of dimension $N$. % $\Ifive$ is the $5\times 5$ identity matrix.
The matrices
$\Pmat_{(\cdot)}$ may be thought of as mass matrices in the context of the \dadd{discontinuous} Galerkin
finite element method. %While it is not true in general that $\Pmat$ is
%diagonal,
Herein, the focus is exclusively on
\dadd{diagonal-norm} SBP operators, based on fixed element-based polynomials of order $p$ ($p=N-1$).
The matrices $\Dmat_{(\cdot)}$ \reftwo{are} used to approximate the first derivatives \reftwo{and} are
defined as $\Pmat^{-1}_{(\cdot)}\Qmat_{(\cdot)}$.
The nearly skew-symmetric matrices $\Qmat_{(\cdot)}$ are undivided differencing operators
where all rows sum to zero, and the first and last column sum to $-1$ and $1$ respectively.
The matrices $\Bmat_{(\cdot)}$ pick \dadd{off} the interface terms in the respective \reftwo{directions}. 
\dadd{For the spectral element discretization considered in this paper, the $\Bmat_{(\cdot)}$ matrices take 
on a particularly simple form; as an example, consider $\Bmat_{x_1}$, which is given as}
\begin{equation*}
\Bmat_{x_{1}} = \Bmat_{x_{1}}^{+}-\Bmat_{x_{1}}^{-},\quad 
\Bmat_{x_{1}}^{-}=\diag\left(1,0,\dots,0\right)\otimes\mat{I}_{N}\otimes\mat{I}_{N}\otimes\mat{I}_{5},\quad
\Bmat_{x_{1}}^{+}=\diag\left(0,\dots,0,1\right)\otimes\mat{I}_{N}\otimes\mat{I}_{N}\otimes\mat{I}_{5}.
\end{equation*}

\reftwo{For} a high-order accurate scheme on a tensor product cell, they pick \dadd{off} the
\dadd{values of whatever vector they act on (typically the solution or the flux)} at the nodes of the two opposite faces \dadd{multiplied by the orthogonal component of the unit normal}.

\reftwo{SBP operators can be recast in telescoping flux form \cite{fisher_phd_2012}.  For example,} 
\reftwo{
\begin{equation*}
\Dxone\fxmIlin{1}=\Pxgen{1}^{-1}\Qxone\fxmIlin{1}=\Pxgen{1}^{-1}\Delta_{x_1}\IStoFxm{1}\fxmIlin{1} =\Pxgen{1}^{-1}\Delta_{x_1}\fxmIlinbar{1},
\end{equation*} 
}
\reftwo{where $\fxmIlinbar{1}=\IStoFxm{1}\fxmIlin{1}$ and the one-dimensional telescoping operator, $\Delta_{N}$, is defined as}
\reftwo{
\begin{equation*}
\label{eq:delta}
 \Delta_N = \left(
 \begin{array}{cccccc}
  -1 & 1 & 0 & 0 & 0 & 0 \\
  0 & -1 & 1 & 0 & 0 & 0 \\
  0 & 0 & \ddots & \ddots & 0 & 0 \\
  0 & 0 & 0 & -1 & 1 & 0 \\
  0 & 0 & 0 & 0 & -1 & 1
 \end{array} 
 \right).
\end{equation*}}

\reftwo{The operator $\IStoFxm{1}$ interpolates the flux at the solution nodes on to a set of flux nodes (see Figure~\ref{fig:fluxnodes}).}
\begin{figure}
  \centering
  \includegraphics[width=0.6\textwidth]{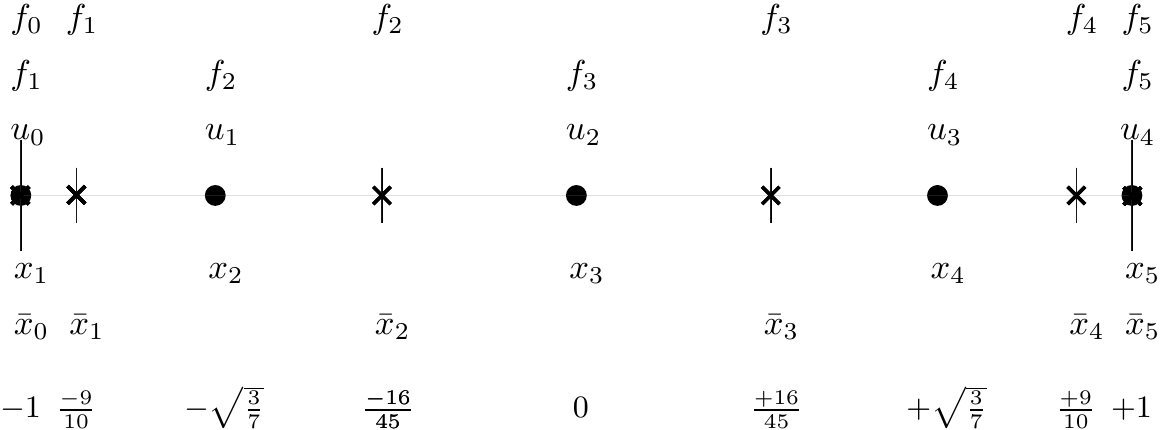}
	\caption{\reftwo{The one-dimensional discretization for $p=4$ Legendre collocation.
	Solution LGL points are denoted by $\bullet$ and flux points are denoted by $\times$.}}
  \label{fig:fluxnodes}
\end{figure}

When applying any of these operators to the scalar entropy equation in space,
a hat will be used to differentiate the scalar operator from the full vector
operator, e.g.
\begin{equation*}
 \widehat{\Pmatvol} = \left(\Poned \otimes \Poned \otimes \Poned \right).
\end{equation*}
\reftwo{We finally note that in the present work, the quadrature nodes and solution nodes are collocated.}
%Note that the spacing between the flux points is
%incorporated into the operator $\Pmat$. In fact, the diagonal elements of $\Pmat$
%are equal to the spacing between flux points.
%In the remainder of this paper,
%the elements of $\Pmat$ are denoted as $\Pmat_{(i)(i)}, i=1,2,\ldots,N$.

Using an SBP operator and its equivalent telescoping form, the
semi-discrete form of the three-dimensional compressible Navier--Stokes equations \eqref{eq:compressible_ns}
for Cartesian grids in each hexahedral element reads (see\dadd{,} e.g. \cite{carpenter_ssdc_2014,parsani_entropy_stability_solid_wall_2015})
\begin{equation}\label{eq:semi-discrete-element}
\begin{aligned}
  \frac{\partial \q}{\partial t} + \sum_{m=1}^3\left(\Pxgeninv{m} \Delta_{x_m} \fxmIlinbar{m} - \Dxgen{m} \fxmVlin{m}\right) =
    \sum_{m=1}^3\Pxgeninv{m} \left(\bm{{g}}_{x_m}^{(B)} + \bm{{g}}_{x_m}^{(In)}\right).
\end{aligned}
\end{equation}
We remark that we have omitted any term involved with the geometrical mapping of the elements in order to simplify the discussion.
\reftwo{The extension of the following analysis to curvilinear grids is
straightforward and all the theorems presented herein \dadd{remain} valid.
The interested reader is referred to the supplementary python script (see 
\ref{app:python_proofs}) which symbolically verifies all the proofs for the general case of curvilinear grids. Furthermore,
a simple, dimension-agnostic implementation of the entropy stable solid wall boundary conditions for the 
general case is
provided in \ref{app:fortran_code}.}

The vectors $\bm{{g}}_{x_m}^{(B)}$ in \eqref{eq:semi-discrete-element}
enforce the boundary conditions, while $\bm{{g}}_{x_m}^{(In)}$ patches
interfaces together using a SAT approach \cite{parsani_entropy_stable_interfaces_2015}.
The derivatives
appearing in the viscous fluxes $\fxmVlin{m}$ are also computed using the
operators $\Dmat_{x_m}$ defined in \eqref{eq:SBP-tensor-matrices}.

The discrete no-slip wall boundary conditions constructed herein follows a local discontinuous Galerkin-type approach
\cite{doi:10.1137/S0036142997316712}, where also the gradient of the entropy variables
is penalized. The penalization of the gradient represents one of the main novelties \reftwo{in} this work
within the context of SBP-SAT discretizations for the imposition of boundary
conditions. In fact, this is a key difference with respect to previous work
\cite{parsani_entropy_stability_solid_wall_2015} which introduced
for the first time an entropy stable approach to impose solid wall boundary
conditions.

Following the same procedure based on local discontinuous Galerkin (LDG) and 
interior penalty approach (IP)
described in \cite{ss-no-slip-wall-bc-parsani-nasa-tm-2014,parsani_entropy_stable_interfaces_2015},
equation \eqref{eq:semi-discrete-element} can be recast as
\begin{subequations}\label{eq:ldg-ip}
\begin{equation}\label{eq:ldg-ip-1}
\begin{aligned}
  \frac{\partial \q}{\partial t} + \sum_{m=1}^3\left(\Pxgeninv{m} \Delta_{x_m} \fxmIlinbar{m} - \sum\limits_{j=1}^{3}\Dxgen{m} \left[\Cij{m}{j}\right] \XiN{x_j}\right) =
    \sum_{m=1}^3\Pxgeninv{m} \left(\bm{{g}}_{x_m}^{(B),q} + \bm{{g}}_{x_m}^{(In),q}\right)
\end{aligned}
\end{equation}
\begin{equation}\label{eq:ldg-ip-2}
  \begin{aligned}
    \XiN{x_m} - \Dxgen{m} \w = \Pxgeninv{m} \left(\bm{{g}}_{x_m}^{(B),\Theta} + \bm{{g}}_{x_m}^{(In),\Theta}\right), \quad m =1,2,3,
\end{aligned}
\end{equation}
\end{subequations}
where $\XiN{x_m}$ \reftwo{is the gradient} of the
entropy variables in the $m$-th direction, whereas
$\bm{{g}}_{x_m}^{(B),q}$, $\bm{{g}}_{x_m}^{(B),\Theta}$ and
$\bm{{g}}_{x_m}^{(In),q}$, $\bm{{g}}_{x_m}^{(In),\Theta}$
are the SAT penalty boundary ($B$) and
interface ($I$) terms on the conservative variable $\q$, and the gradient of the entropy
variable $\bm{\Theta}$, respectively \cite{ss-no-slip-wall-bc-parsani-nasa-tm-2014}. The
contributions of the interface
penalty terms are non-zero only in the normal direction to the interface.
The matrices $\left[\Cij{m}{j}\right]$ are block diagonal matrices with blocks of size 5,
corresponding to the viscous coefficients at each solution point.

\begin{remark}\label{rm:ho_ss_flux-1}
\dadd{In order to build a high-order accurate entropy-conservative/stable spatial discretization, the linear interpolation 
operation, $\IStoFxm{m}$ is replaced with a non-linear interpolation operator (or equivalently, the linear SBP operator is 
replaced with a non-linear SBP operator)} \cite{tadmor_review_entropy_analysis_2003,fisher_phd_2012,carpenter_ssdc_2014,
  parsani_entropy_stability_solid_wall_2015}. \dadd{Critically, the resulting non-linear operator, 
$\Pxgen{m}^{-1}\Delta_{\xm}\fxmIlinbarsc{m}$ or $\Pxgen{m}^{-1}\Delta_{\xm}\fxmIlinbarssr{m}$ (entropy conservative or entropy stable, respectively), has the property that when contracted with the entropy variables and 
discretely integrated over the domain, the result is a discrete surface integral 
with respect to the entropy flux; that is, for the entropy-conservative formulation, the operator telescopes in the entropy flux. For example
\begin{equation*}
\bm{w}\Tr\mat{P}\Pxgen{1}^{-1}\Delta_{x_1}\fxmIlinbarsc{1}=\bm{1}\Tr\left(\Pxgentwodhat{2}{3} \Bxgenhat{1} \bm{F}_{x_1}\right)
\approx\int_{x_{1}=1}\fnc{F}_{x_{1}}\mr{d}x_{2}\mr{d}x_{3}-\int_{x_{1}=0}\fnc{F}_{x_{1}}\mr{d}x_{2}\mr{d}x_{3},
\end{equation*}
where $\fxmIlinbarsc{1}$ is the vector that results from the entropy-conservative non-linear interpolation of the flux (this flux is replaced with $\fxmIlinbarssr{1}$ for the entropy-stable version), $\bm{1}$ is a vector of ones of appropriate size, and $\bm{F}_{x_1}$ is a vector of the entropy flux in the $x_1$ direction 
evaluated on the mesh nodes. 
}
\end{remark}

To obtain an equation for the entropy of the system, we follow the
entropy stability analysis presented in
\cite{carpenter_ssdc_2014,ss-no-slip-wall-bc-parsani-nasa-tm-2014}. Therefore,
multiplying the two discrete equations by
$\w^{\top}\Pmatvol$ and
$\left(\left[\Cij{m}{j}\right]\XiN{x_j}\right)^{\top}\Pmatvol$, respectively,
the expression for the time derivative of the entropy function  in each
element is
\begin{equation}\label{eq:estimate-no-slip-bc-1}
  \begin{aligned}
    \frac{d}{dt} & \mathbf{1}^{\top} \widehat{\Pmatvol} \, \bm{\fnc{S}} + \,  \mathbf{DT}
    \:+\: \mathbf{1}^{\top} \left(\Pxgentwodhat{2}{3} \Bxgenhat{1} \bm{F}_{x_1}
				+ \Pxgentwodhat{1}{3} \Bxgenhat{2} \bm{F}_{x_2}
				+ \Pxgentwodhat{1}{2} \Bxgenhat{3} \bm{F}_{x_3}  \right) \\
    & = \,\w^{\top} \left(\Pxgentwod{2}{3} \Bxone   \left[\Cij{1}{j}\right] \XiN{x_j}
                        + \Pxgentwod{1}{3} \Bxtwo   \left[\Cij{2}{j}\right] \XiN{x_j}
                        + \Pxgentwod{1}{2} \Bxthree \left[\Cij{3}{j}\right] \XiN{x_j} \right)  \\
    & + \,\w^{\top} \left(\Pxgentwod{2}{3} \left(\bm{{g}}_{x_1}^{(B),q} + \bm{{g}}_{x_1}^{(In),q}\right)
                        + \Pxgentwod{1}{3} \left(\bm{{g}}_{x_2}^{(B),q} + \bm{{g}}_{x_2}^{(In),q} \right)
			+ \Pxgentwod{1}{2} \left(\bm{{g}}_{x_3}^{(B),q} + \bm{{g}}_{x_3}^{(In),q} \right)\right) \\
    & + \left(\left[\Cij{1}{j}\right]\XiN{x_j}\right)^{\top} \Pxgentwod{2}{3} \left(\bm{{g}}_{x_1}^{(B),\Theta} + \bm{{g}}_{x_1}^{(In),\Theta}\right)
      + \left(\left[\Cij{2}{j}\right]\XiN{x_j}\right)^{\top} \Pxgentwod{1}{3} \left(\bm{{g}}_{x_2}^{(B),\Theta} + \bm{{g}}_{x_2}^{(In),\Theta}\right) \\
    & + \left(\left[\Cij{3}{j}\right]\XiN{x_j}\right)^{\top} \Pxgentwod{1}{2} \left(\bm{{g}}_{x_3}^{(B),\Theta} + \bm{{g}}_{x_3}^{(In),\Theta}\right),
\end{aligned}
\end{equation}
where $$\mathbf{DT} = \left\|\sqrt{[\Cij{m}{j}]} \, \XiN{j} \right\|^2_{\Pmatvol}\dadd{,}$$
is a positive quadratic term in the approximation of the first
derivative of the solution \cite{parsani_entropy_stability_solid_wall_2015}, $\mathbf{1}$ is the unit vector of appropriate size, and
\dadd{$\bm{F}_{x_m}$} the vector of the entropy flux in the $m$-th direction. 

Equation \eqref{eq:estimate-no-slip-bc-1} can be conveniently rewritten as 
\begin{equation}\label{eq:estimate-no-slip-bc-2}
    \frac{d}{dt} \mathbf{1}^{\top} \widehat{\Pmatvol} \, \bm{\fnc{S}} + \,  \mathbf{DT}
	= \mathbf{\Xi}. %\mathbf{\Gamma}_1 + \mathbf{\Gamma}_2 + \mathbf{\Gamma}_3,
\end{equation}
This compact notation will \dadd{be} useful in Section \ref{sec:numerical_results}.

As \dadd{in the continuous analysis}, we assume that we have a hexahedral element with edge length
equal to one, and consider only the plane $(0,x_2,x_3)$ as wall boundary face. Thus, we have 
$\bm{{g}}_{x_m}^{(In),q}=\bm{{g}}_{x_m}^{(In),\Theta}=0, ~\forall m$. We also assume that all the 
points that lie on the other faces of the cube are treated in an entropy stable fashion such that 
their contribution can be neglected. \dadd{With these assumptions,} equation \eqref{eq:estimate-no-slip-bc-1} reduces to
\dadd{
\begin{equation}\label{eq:estimate-no-slip-bc-1-face}
  \begin{aligned}
    \frac{d}{dt} \mathbf{1}^{\top} \widehat{\Pmatvol} \, \bm{\fnc{S}} & + \,
    \mathbf{DT}
    \:+\: \mathbf{1}^{\top} \Pxgentwodhat{2}{3} \widehat{\mat{B}}_{x_1}^{-} \bm{F}_{x_1}
     = \,\w^{\top} \Pxgentwod{2}{3} \mat{B}_{x_{1}}^{-} \left[\Cij{1}{j}\right] \XiN{x_j} \\
     & + \,\w^{\top} \Pxgentwod{2}{3} \bm{{g}}_{x_1}^{(B),q}
       + \left(\left[\Cij{1}{j}\right]\XiN{x_j}\right)^{\top} \Pxgentwod{2}{3} \bm{{g}}_{x_1}^{(B),\Theta}.
\end{aligned}
\end{equation}
}
The penalty source term $\bm{{g}}_{x_1}^{(B),q}$ is composed of
three design-order terms plus a source boundary term
\begin{equation}\label{eq:SAT-variables}
  \begin{aligned}
    \bm{{g}}_{x_1}^{(B),q} =  \bm{{g}}_{x_1}^{(B,I),q} + \bm{{g}}_{x_1}^{(B,V),q} + \ensuremath{\bm{\fnc{M}}}^{(B,V)} + \ensuremath{\bm{\fnc{L}}^{(B,V)}},
    %- \left(
    %\Fxgen{1} - \bm{f}^{sc}\left(\ensuremath{\bm{v}},\bm{v}^{(B,I)} \right)\right) %\\
    %+ \frac{1}{2}\left(\left[\Cij{1}{j}\right]\XiN{x_j} - \ensuremath{\bm{\fnc{F}}_{x_{1}}^{(B,V)}} \right) + \ensuremath{\bm{\fnc{L}}^{(B,V)}},
\end{aligned}
\end{equation}
where
\dadd{
\begin{equation}\label{eq:g_BI_q}
\bm{{g}}_{x_1}^{(B,I),q} = -\mat{B}_{x_{1}}^{-} \left(\bm{f}_{x_1} - \bm{f}^{sc}\left(\ensuremath{\bm{v}},\bm{v}^{(B,I)} \right)\right)\dadd{,}
\end{equation}
}
\dadd{$\bm{f}^{sc}\left(\ensuremath{\bm{v}},\bm{v}^{(B,I)} \right)$ is the entropy conservative flux vector, which 
is a function of the two arguments $\bm{v}$ and $\bm{v}^{(B,I)}$,} and
\dadd{
\begin{equation}\label{eq:g_BV_q}
\bm{{g}}_{x_1}^{(B,V),q} = \frac{1}{2}\mat{B}_{x_{1}}^{-} \left(\left[\Cij{1}{j}\right]\XiN{x_j} - \ensuremath{\bm{f}_{x_{1}}^{(B,V)}} \right).
\end{equation}
}
For clarity of presentation, the expressions of $\ensuremath{\bm{\fnc{M}}}^{(B,V)}$, which is used to add provably dissipation,
and of $\ensuremath{\bm{\fnc{L}}^{(B,V)}}$ will be given and analyzed at the end of the Section.

The penalty $\bm{{g}}_{x_1}^{(B),\Theta}$ \reftwo{contains} a single design-order term
\dadd{
\begin{equation}\label{eq:g_BV_grad}
  \begin{aligned}
    \bm{{g}}_{x_1}^{(B),\Theta} %=  \bm{{g}}_{x_1}^{(B,V),\Theta}
    = \frac{1}{2}\mat{B}_{x_{1}}^{-}\left(\w-\w^{(B,V)}\right).
\end{aligned}
\end{equation}
}
In each of the contributions, the first component (the numerical state)
is constructed from the numerical solution, while the second component (the boundary state)
is constructed from a combination of the numerical solution and
four independent components of physical boundary data.

In what follows, we analyze
each of these contributions \reftwo{from \eqref{eq:estimate-no-slip-bc-1-face} separately}. We further restrict our analysis to a single solution point lying on the wall boundary face. Therefore, we will no longer make use of the bold notation, which will be replaced by italics to denote the vector of \reftwo{five} components at the collocated point.

The term given in \eqref{eq:g_BI_q} %, i.e. $\bm{{g}}_{x_1}^{(B,I),q}$,
enforces the Euler no-penetration wall
condition through the inviscid flux of the compressible Euler equations.
The boundary state is formed by constructing an entropy conservative flux based on the
numerical state in primitive variables at the face point\dadd{,} $v$\dadd{,}
and a manufactured boundary state given by the
vector of the primitive variables
\begin{equation}\label{eq:definition-gIB}
  v^{(B,I)} =
%  \begin{pmatrix}
%    1 & 0 & 0 & 0 & 0 \\
%    0 &-1 & 0 & 0 & 0 \\
%    0 & 0 & 1 & 0 & 0 \\
%    0 & 0 & 0 & 1 & 0 \\
%    0 & 0 & 0 & 0 & 1
%\end{pmatrix}
\textrm{diag}([1,-1,1,1,1]), \quad
v  =
\left(\rho,
  -{\Uone},
  {\Utwo},
  {\Uthree},
  {\fnc{T}} \right)^{\top}.
\end{equation}

The term defined in \eqref{eq:g_BV_q}, %i.e. $\bm{{g}}_{x_1}^{(B,V),q}$,
together with the penalty \eqref{eq:g_BV_grad}, % $\bm{{g}}_{x_1}^{(B),\Theta}$
allow a weak imposition of the no-slip condition
\[
\fnc{U}_m=\fnc{U}^{wall}_m, \quad \fnc{U}^{wall}_1=0,\quad\dadd{m=2,3,}
\]
in an entropy conservative way, where
$\fnc{U}^{wall}_m$ is the $m$-th component of the wall velocity.
Furthermore, entropy flow is enforced 
if $\kappa\left(\frac{\partial \fnc{T}}{\partial x_1} \frac{1}{\fnc{T}}\right) \neq 0$.

Using the
following primitive variables
\begin{equation}\label{eq:definition-V-for-wB}
\begin{aligned}
  v^{(B,V)} =
%  \begin{pmatrix}
%    1 & 0 & 0 & 0 & 0 \\
%    0 &-1 & 0 & 0 & 0 \\
%    0 & 0 & -1 & 0 & 0 \\
%    0 & 0 & 0 & -1 & 0 \\
%    0 & 0 & 0 & 0 & 1
%\end{pmatrix}
%  v_{(1)}^{\top} & =
\left(\rho,
  - \ensuremath{\fnc{U}_{1}},
  - \ensuremath{\fnc{U}_{2}} + 2 \, \ensuremath{\fnc{U}_{2}^{wall}},
  - \ensuremath{\fnc{U}_{3}} + 2 \, \ensuremath{\fnc{U}_{3}^{wall}},
  \fnc{T} \right)^{\top},
\end{aligned}
\end{equation}
we define the point-wise boundary viscous flux $\ensuremath{{\fnc{F}}}_{x_{1}}^{(B,V)}$,
the term $\ensuremath{{\fnc{M}}}^{(B,V)}$,
and source term $\ensuremath{{\fnc{L}}}^{(B,V)}$  as
\begin{equation}
\ensuremath{{\fnc{F}}}_{x_{1}}^{(B,V)} = \Cij{1}{j}^{(B,V)} \, \XiNnobold{x_j}^{(B,V)},
  %\quad \mathtt{g}(t) \equiv \textrm{bounded boundary datum} \, (L^2 \cap L^{\infty}),
\end{equation}
\begin{equation}\label{eq:dissipation_def}
  \ensuremath{{\fnc{M}}}^{(B,V)} = L \left(w-w^{(B,V)}\right),
\end{equation}
\begin{equation}
\ensuremath{{\fnc{L}}}^{(B,V)} = - (0,0,0,0,1)^{\top} \, \fnc{T} \mathtt{g}(t),
\end{equation}
where $\mathtt{g}(t) = \kappa \frac{\partial \fnc{T}}{\partial x_1} \frac{1}{\fnc{T}}$ is a given $L^2$
function. \refone{Thus, in case of a boundary condition with imposed non-zero heat flux, $\mathtt{g}(t)$ is non-zero.}  We remark that the vector  $v^{(B,V)}$ is used to evaluate the matrix of the viscous coefficients $\Cij{1}{j}^{(B,V)}$, as well as to compute the penalty term in \eqref{eq:g_BV_grad}.

The manufactured gradient of the entropy variables at the boundary $\Theta_{x_j}^{(B,V)}$
is constructed using the following procedure:
\begin{itemize}\label{item:proecdure_grad}
\item Rotate the gradient of the entropy variables $\Theta_{x_j}$ to
  the gradient of the primitive variables, $v$, %$\Pi_{x_j}$:
    \begin{equation}
    \Pi_{x_j} = \frac{\partial{V}}{\partial{W}} \Theta_{x_j}, \quad j = 1,2,3,
  \end{equation}
    where $\frac{\partial{V}}{\partial{W}}$ is the Jacobian of the transformation  in primitive variables
    with respect to the entropy variables evaluated at the face point.
\item Construct the gradient of the primitive variables at the wall boundary point %as $\Pi^{(B,V)}_{x_j}$
  \begin{equation}
  \begin{aligned}
%    & \Pi^{(B,V)}(1,:) = -\Pi(1,:)_{(1)} \\
%    & \Pi^{(B,V)}(2,:) = +\Pi(2,:)_{(1)} \\
%    & \Pi^{(B,V)}(3,:) = +\Pi(3,:)_{(1)} \\
%    & \Pi^{(B,V)}(4,:) = +\Pi(4,:)_{(1)} \\
%    & \Pi^{(B,V)}(5,:) = -\Pi(5,:)_{(1)}
    \Pi^{(B,V)}_{x_j} = \textrm{diag}([-1,1,1,1,-1]) \, \Pi_{x_j}, \quad j=1,2,3.
  \end{aligned}
  \end{equation}
  \refone{This choice \dadd{follows} the imposition of Neumann
    boundary conditions in the context of the nodal discontinuous Galerkin method
    \cite{hesthaven_2008_nodal_dg}.}
\item Rotate the gradient of the primitive variables $\Pi_{x_j}^{(B,V)}$ to the gradient
  of the entropy variables %$\Theta^{(B,V)}_{x_j}$:
  \begin{equation}
    \Theta^{(B,V)}_{x_j} = \frac{\partial{W}}{\partial{V}}\bigg\rvert_{(B)} \Pi^{(B,V)}_{x_j}, \quad j = 1,2,3,
  \end{equation}
    where $\frac{\partial{W}}{\partial{V}}\big\rvert_{(B)}$ is the Jacobian of the entropy variables
    with respect to the primitive variables evaluated using the state $v^{(B,V)}$ defined in \eqref{eq:definition-V-for-wB}.
\end{itemize}
Finally, the entropy variables $w^{(B,V)}$ needed in \eqref{eq:g_BV_grad} are computed from the primitive variables
defined in expression \eqref{eq:definition-V-for-wB} by using the relations
in \eqref{eq:entropy_variables}.

The matrix $L$ in \eqref{eq:dissipation_def} is a negative semi-definite
5x5 matrix which is defined as
\begin{equation}\label{eq:def_L_mat_diss}
  L = -\beta \frac{\Cij{1}{1} + \Cij{1}{1}^{(B,V)}}{2},
\end{equation}
where $\Cij{1}{1}$ and $\Cij{1}{1}^{(B,V)}$ are the positive
semi-definite viscous coefficient matrices
in the normal direction ($m=j=1$), respectively evaluated using the states $v$ and $v^{(B,V)}$,
and $\beta$ is a positive coefficient that modulates
the strength of the penalty term. This coefficient has to scale as the inverse of the
typical \dadd{element} length\footnote{\reftwo{This \dadd{is} done such that the scaling remains dimensionally consistent.}}.

Summarizing, the penalty at the face point for the conservative variables $\q$ is the sum
of two terms:
\begin{itemize}
  \item the difference between inviscid and entropy conservative fluxes
     in the normal direction,
  \item the difference between internal viscous and boundary
    viscous fluxes in the normal direction.
\end{itemize}
The penalty on the gradient of the entropy variables $\bm{\Theta}$ is instead given by the difference between the solution at the node
and the data imposed at the boundary expressed in terms of entropy variables.

The entropy conservation and stability of the penalty source terms \eqref{eq:SAT-variables}
and \eqref{eq:g_BV_grad} is demonstrated in the following
three theorems. The first theorem, which ensures entropy conservation for the inviscid
SAT penalty in \eqref{eq:SAT-variables}, and which enforces the no penetration condition,
is Theorem 5.1 in \cite{parsani_entropy_stability_solid_wall_2015}. The second\dadd{,} 
is a new theorem and it ensures entropy conservation or stability for the viscous SAT penalty.
The third theorem is\dadd{ also} new and ensures that the term $\bm{\fnc{M}}^{(B,V)}$ is a \reftwo{dissipative entropy
contribution.}

\begin{theorem}\label{th:euler-flipping-sign-1}
The penalty inviscid flux contribution in equation
\eqref{eq:SAT-variables} is entropy conservative if the
vector $v^{(B,I)}$ is defined as in \eqref{eq:definition-gIB}.
\end{theorem}
\begin{proof}
See Theorem 5.1 in \cite{parsani_entropy_stability_solid_wall_2015}.
\end{proof}

\begin{theorem}\label{th:ns-viscous-sat}
The penalty terms for the viscous flux on the conserved variables \eqref{eq:SAT-variables},
together with the viscous penalty on the gradient of the entropy variables \eqref{eq:g_BV_grad},
  are
  \begin{itemize}
  \item entropy conservative if the wall is adiabatic,  i.e. $\mathtt{g}(t) =0$,
  \item entropy stable  in the presence of a heat flux, i.e. $\mathtt{g}(t) \neq 0$, where $\mathtt{g}(t)$ is a given $L^2$ function.
  \end{itemize}
\end{theorem}
\begin{proof}
By substituting \dadd{into} \eqref{eq:estimate-no-slip-bc-1-face} the expressions
  for $\bm{{g}}_{x_1}^{(B),q}$ with $\ensuremath{\bm{\fnc{M}}}^{(B,V)}=0$ (i.e. no dissipation) and $\bm{{g}}_{x_1}^{(B),\Theta}$ given in
\eqref{eq:SAT-variables} and \eqref{eq:g_BV_grad}, respectively, yields
\dadd{
  \[
    %\frac{d}{dt} \mathbf{1}^{\top} \widehat{\Pmatvol} \, \bm{\fnc{S}} + \left\|\sqrt{[\Cij{1}{j}]} \, \XiN{j} \right\|^2_{\Pmatvol} = \Pxgentwodhat{2}{3} \, \mathtt{g}(t).
    \frac{d}{dt} \mathbf{1}^{\top} \widehat{\Pmatvol} \, \bm{\fnc{S}} +  \mathbf{DT} =\mathbf{1}^{\top} \Pxgentwodhat{2}{3} \, \mathtt{g}(t).
  \]
}
For an adiabatic wall $\mathtt{g}(t) = 0$ and therefore the proposed boundary
conditions are entropy conservative. For $\mathtt{g}(t) \neq 0$ the boundary
conditions are entropy stable because the contribution to the time rate of change
of the entropy function is only a function of the data $\mathtt{g}(t)$.
\end{proof}

\begin{theorem}\label{th:ns-viscous-ip-dissipation}
The interior penalty term
  \begin{equation}
    \bm{\fnc{M}}^{(B,V)} = [L] \left(\w-\w^{(B,V)}\right)
  \end{equation}
added to the SAT \eqref{eq:SAT-variables} is entropy dissipative.
\end{theorem}
\begin{proof}
  By expanding the contraction $\w^{\top} \Pxgentwod{2}{3} \bm{{g}}_{x_1}^{(B),q}$
  in \eqref{eq:estimate-no-slip-bc-1-face}, and by focusing on the dissipation term only,
  we arrive at the following point-wise contribution to the time-rate of change
  of the entropy function
  \begin{equation}\label{eq:IP-dissipation-proof-1}
      w^{\top}  L \left(w-w^{(B,V)}\right).
  \end{equation}
  Note that we have omitted an extra positive scaling factor corresponding to the entry of the matrix $\Pxgentwod{2}{3}$ associated with a wall boundary point.
  Plugging in the definitions of the matrix $L$ \eqref{eq:def_L_mat_diss}, $w$ and $w^{(B,V)}$
  %\[
  %  \frac{2 \beta\mu}{3 \fnc{T}}\left(-4 \fnc{U}_{1,(1)}^2 - 3 \fnc{U}_2^2 + 6 \fnc{U}_2 \fnc{U}_2^{wall} - 3 \fnc{U}_2^2 + 6 \fnc{U}_2 \fnc{U}^{wall}_2 - 3 \left(\fnc{U}_1^{wall}\right)^2 - 3 \left(\fnc{U}_2^{wall}\right)^2\right).
  %\]
  we obtain
  \[
    w^{\top}  L \left(w-w^{(B,V)}\right) = - \frac{2 \beta\mu}{3 \fnc{T}}\left(4 \, \fnc{U}_{1}^2 + 3 \left(\fnc{U}_{2}^2 - \fnc{U}_2^{wall}\right)^2 + 3 \left(\fnc{U}_{3}^2 - \fnc{U}_3^{wall}\right)^2\right),
  \]
  which completes the proof.
 \end{proof}

\begin{remark}\label{rm:dissipation}
  In order to construct an entropy stable solid wall boundary condition for the Euler
  equations, it is necessary to add \reftwo{provable entropy dissipation.} %to estimate \eqref{eq:estimate-no-slip-bc-1-face}.
  One solution consists of replacing the entropy conservative flux $\bm{f}^{sc}$ in
  \eqref{eq:SAT-variables} with an entropy stable
  flux $\bm{f}^{ssr}$ in \eqref{eq:SAT-variables}
  as described in \cite{parsani_entropy_stability_solid_wall_2015}.
\end{remark}

%%%%%%%%%%%%%%%%%%%%%%%%%%%%%%%%%%%%%%%%%%%%%%%%%%%%%%%%%%%%%%%%%%%%%%%%%%%%%%%
\section{A common SAT procedure for the imposition of wall boundary conditions and interior interface coupling}
\dadd{
The proposed approach for imposing the solid wall
boundary conditions allows for a SAT implementation which is identical
to the interface treatment shown in \cite{parsani_entropy_stable_interfaces_2015}.
We can use a single subroutine with different inputs corresponding to the imposition of
the interior interface couplings, or of the adiabatic solid wall or of the wall with a prescribed heat entropy flow. In fact,
the interior interface coupling can be written as (see equations (16a-16d) in \cite{parsani_entropy_stable_interfaces_2015})
\begin{subequations}\label{eq:left-right-ldg-ip}
\begin{equation}\label{eq:left-ldg-ip-1}
\begin{aligned}
	\frac{\partial \q_l}{\partial t} + \sum_{m=1}^3\left({\Pxgeninv{m,l}} \, {\Delta_{x_m,l}} \, {\fxmIlinbar{m,l}} - \sum\limits_{j=1}^{3}\Dxgen{m,l} \left[\Cij{m}{j,l}\right] \XiN{x_j,l}\right) =
    \sum_{m=1}^3\Pxgeninv{m,l} \, \bm{{g}}_{x_m,l}^{(In),q},
\end{aligned}
\end{equation}
\begin{equation}\label{eq:left-ldg-ip-2}
  \begin{aligned}
    \XiN{x_m,l} - \Dxgen{m,l} \w = \Pxgeninv{m,l} \, \bm{{g}}_{x_m,l}^{(In),\Theta}, \quad m =1,2,3,
\end{aligned}
\end{equation}
\begin{equation}\label{eq:right-ldg-ip-1}
\begin{aligned}
	\frac{\partial \q_r}{\partial t} + \sum_{m=1}^3\left({\Pxgeninv{m,r}} \, {\Delta_{x_m,r}} \, {\fxmIlinbar{m,r}} - \sum\limits_{j=1}^{3}\Dxgen{m,r} \left[\Cij{m}{j,r}\right] \XiN{x_j,r}\right) =
    \sum_{m=1}^3\Pxgeninv{m,r} \, \bm{{g}}_{x_m,r}^{(In),q},
\end{aligned}
\end{equation}
\begin{equation}\label{eq:right-ldg-ip-2}
  \begin{aligned}
    \XiN{x_m,r} - \Dxgen{m,r} \w = \Pxgeninv{m,r} \, \bm{{g}}_{x_m,r}^{(In),\Theta}, \quad m =1,2,3,
\end{aligned}
\end{equation}
\end{subequations}
which have exactly the same structure as LDG-IP approach used for the imposition of the solid wall boundary conditions except for the boundary penalty interface terms, $\bm{{g}}_{x_m,r}^{(B),\cdot}$ in equation \eqref{eq:ldg-ip}, which are replaced by the interior penalty interface coupling terms, $\bm{{g}}_{x_m,r}^{(In),\cdot}$ in equations \eqref{eq:left-right-ldg-ip}.
}

%%%%%%%%%%%%%%%%%%%%%%%%%%%%%%%%%%%%%%%%%%%%%%%%%%%%%%%%%%%%%%%%%%%%%%%%%%%%%%%%
\section{Numerical results}\label{sec:numerical_results}

In this section we present four three-dimensional test cases which demonstrate the robustness
of the new wall boundary conditions coupled with
the family of high-order accurate entropy-stable interior SBP-SAT algorithms developed in
\cite{carpenter_ssdc_2014,parsani_entropy_stability_solid_wall_2015,parsani_entropy_stable_interfaces_2015,carpenter_entropy_stability_ssdc_2016}.
The systems of ordinary differential equations arising from the spatial
discretizations are integrated using the fourth-order
accurate Dormand--Prince method \cite{dormand_rk_1980} endowed with an adaptive time stepping technique based on digital signal processing \cite{Soderlind2003,Soderlind2006}. \reftwo{We note that} small enough tolerances are always used to make the temporal error negligible.

The unstructured grid solver used herein has been developed at the Extreme Computing Research Center (ECRC) at KAUST on top of the Portable and Extensible Toolkit for Scientific computing (PETSc)~\cite{petsc-user-ref}, its mesh topology abstraction (DMPLEX)~\cite{KnepleyKarpeev09} and scalable ordinary differential equation (ODE)/differential algebraic equations (DAE) solver library~\cite{abhyankar2018petsc}, and the Message Passing Interface (MPI). \reftwo{Additionally,} the numerical solver is based on the algorithms proposed in \cite{carpenter_ssdc_2014,parsani_entropy_stability_solid_wall_2015,parsani_entropy_stable_interfaces_2015,carpenter_entropy_stability_ssdc_2016}.
It uses a transformation from computational to physical space that satisfies both the entropy conservation and the
geometric conservation law at the semi-discrete level \cite{fisher_phd_2012}.
Unless otherwise stated, the meshes used in this work have been generated using the GMSH package \cite{geuzaine2009gmsh}.

\dadd{We present the numerical results for five test cases:
\begin{itemize}
\item Laminar flow in a pipe with annular section to verify the accuracy of procedure for the imposition of the solid wall boundary conditions; 
\item Laminar flow in a lid-driven cavity to validate the entropy conservation and entropy stability properties of the interior discretization operator coupled with the boundary procedure;
\item Laminar flow past a three-dimensional cylinder and a sphere to demonstrate the engineering capabilities of the interior discretization operator coupled with the boundary procedure;
\item Supersonic turbulent flow past a three-dimensional rod with square section to demonstrate the robustness of the solver and solid wall boundary conditions (``standard'' SBP-SAT operators and procedure for imposing solid wall boundary conditions based on linear analysis crash).
\end{itemize}}

\dadd{
\subsection{Flow in pipe with annular cross-section}
In this section we investigate the accuracy of the solid wall boundary conditions. 
The proposed entropy stable no-slip wall boundary conditions do not force the numerical 
solution to exactly fulfill the boundary conditions. Instead the effect can be described as 
a rubber-band pulling the solution towards the boundary conditions. The computed boundary value
(or numerical state) typically deviates slightly from the prescribed value but the deviation is 
reduced as the grid is refined. To verify the accuracy of the new procedure, we perform a grid convergence
study for the flow in a full three-dimensional pipe with annular cross-section. For an incompressible flow, this test case has an analytical solution \cite{rosenhead_bl_book} for both the velocity distribution and the volume flux through the annular pipe. The expression for the axial velocity distribution, $\Uone$, as a function of the radial coordinate, $r$, is
\begin{equation}
	\ensuremath{\fnc{U}_{1}}(r) = \frac{G}{4\mu} \left[\left(R_{i}^2-r^2\right) + \left(R_{o}^2-R_{i}^2\right) \frac{ln(r/R_{i})}{ln(R_{o}/R_{i})} \right],  
\end{equation}
where $G$ is the pressure gradient forcing term and $R_{i}$ and $R_{o}$ are the inner and outer radii of the pipe, respectively. Herein, we set $R_{o} = 0.5$, $R_{o}/R_{i} = 4$ and $G/\mu=1$.
We highlight that we have chosen this test problem because i) it has an analytical solution which cannot be represented exactly by the polynomial space of the numerical solution\dadd{, and} ii) it requires \dadd{the use of} curved boundary element faces to capture accurately the geometry of the annular section of the pipe.

The code that is used is a compressible code and in order to obtain results that are very close to those found for the incompressible 
equations, a Mach number of $M = 1.0e-05$ is considered. Periodic boundary condition are used in the axial direction.

We run a grid convergence study for $p=2,3,4$ with a sequence of nested grids generated using rational Bezier basis functions
such that the geometrical description of the pipe is preserved exactly for each refinement level.

The error in the axial velocity profile are computed using discrete norms as follows:
\[
\begin{split}
&\text{Discrete }L^{1}: \|\bm{u}\|_{L^{1}}=\sum\limits_{\kappa=1}^{K}\onek\Tr\Mk\matJk\textrm{abs}\left(\uk\right),\\
&\text{Discrete }L^{2}: \|\bm{u}\|_{L^{2}}^2=\sum\limits_{\kappa=1}^{K}\uk\Tr\Mk\matJk\uk,\\
&\text{Discrete }L^{\infty}: \|\bm{u}\|_{L^{\infty}}=\max\limits_{\kappa=1\dots K}\textrm{abs}\left(\uk\right),
\end{split}
\]
where $\matJk$ is the metric Jacobian of the curvilinear transformation from physical
space to computational space of the $k$-th hexahedral element and $K$ is the
total number of non-overlapping hexahedral elements in the mesh.

The results of the grid convergence study are shown in Tables \ref{tab:an_p2}, \ref{tab:an_p3} and \ref{tab:an_p4} where the numbering the first column indicates the number of elements in the radial, angular and axial coordinates. It can be seen that the computed order of accuracy is very close to the formal value of $\sim (p+1)$.

\begin{table}[htbp]
\vspace{0.5cm}
\begin{center}
\begin{tabular}{||l||c|c|c|c|c|c||}
\hline \hline
 Grid & $L^1$    & Rate   & $L^2$     & Rate  & $L^{\infty}$ & Rate  \\ \hline
 4    & 8.77e-02 & -      & 1.79e-01  & -     & 5.69e-01     & -     \\ \hline
 8    & 1.28e-02 & -2.77  & 3.45e-02  & -2.38 & 1.50e-01     & -1.92 \\ \hline
 16   & 2.03e-03 & -2.66  & 5.47e-03  & -2.66 & 3.12e-02     & -2.27 \\ \hline
 32   & 2.68e-04 & -2.92  & 7.51e-04  & -2.86 & 5.07e-03     & -2.62 \\ \hline
 64   & 3.33e-05 & -3.01  & 9.66e-05  & -2.96 & 7.23e-04     & -2.81 \\ \hline
\end{tabular}
\caption{Convergence study for the flow in a pipe with annular section; $p=2$; error in the axial velocity.}.
\label{tab:an_p2}
\end{center}
\end{table}

\begin{table}[htbp]
\vspace{0.5cm}
\begin{center}
\begin{tabular}{||l||c|c|c|c|c|c||}
\hline \hline
 Grid & $L^1$    & Rate   & $L^2$     & Rate  & $L^{\infty}$ & Rate  \\ \hline
 4    & 2.04e-02 & -      & 3.87e-02  & -     & 1.54e-01     & -     \\ \hline
 8    & 2.71e-03 & -2.91  & 5.68e-03  & -2.77 & 3.17e-02     & -2.28 \\ \hline
 16   & 2.15e-04 & -3.65  & 5.69e-04  & -3.32 & 4.43e-03     & -2.84 \\ \hline
 32   & 1.25e-05 & -4.10  & 4.40e-05  & -3.69 & 4.63e-04     & -3.26 \\ \hline
 64   & 6.65e-07 & -4.23  & 2.87e-06  & -3.94 & 4.05e-05     & -3.51 \\ \hline
\end{tabular}
\caption{Convergence study for the flow in a pipe with annular section; $p=3$; error in the axial velocity.}.
\label{tab:an_p3}
\end{center}
\end{table}

\begin{table}[htbp]
\vspace{0.5cm}
\begin{center}
\begin{tabular}{||l||c|c|c|c|c|c||}
\hline \hline
 Grid & $L^1$    & Rate   & $L^2$     & Rate  & $L^{\infty}$ & Rate  \\ \hline
 4    & 4.54e-03 & -      & 9.55e-03  & -     & 4.70e-02     & -     \\ \hline
 8    & 4.48e-04 & -3.34  & 9.58e-04  & -3.32 & 6.31e-03     & -2.90 \\ \hline
 16   & 1.89e-05 & -4.57  & 5.68e-05  & -4.08 & 5.14e-04     & -3.66 \\ \hline
 32   & 5.98e-07 & -4.98  & 2.34e-06  & -4.60 & 2.76e-05     & -4.22 \\ \hline
 64   & 2.22e-08 & -4.75  & 7.65e-08  & -4.95 & 1.13e-06     & -4.61 \\ \hline
\end{tabular}
\caption{Convergence study for the flow in a pipe with annular section; $p=4$; error in the axial velocity.}.
\label{tab:an_p4}
\end{center}
\end{table}

}

\subsection{Lid-driven cavity}

\begin{figure}
  \centering
  \includegraphics[width=0.6\textwidth]{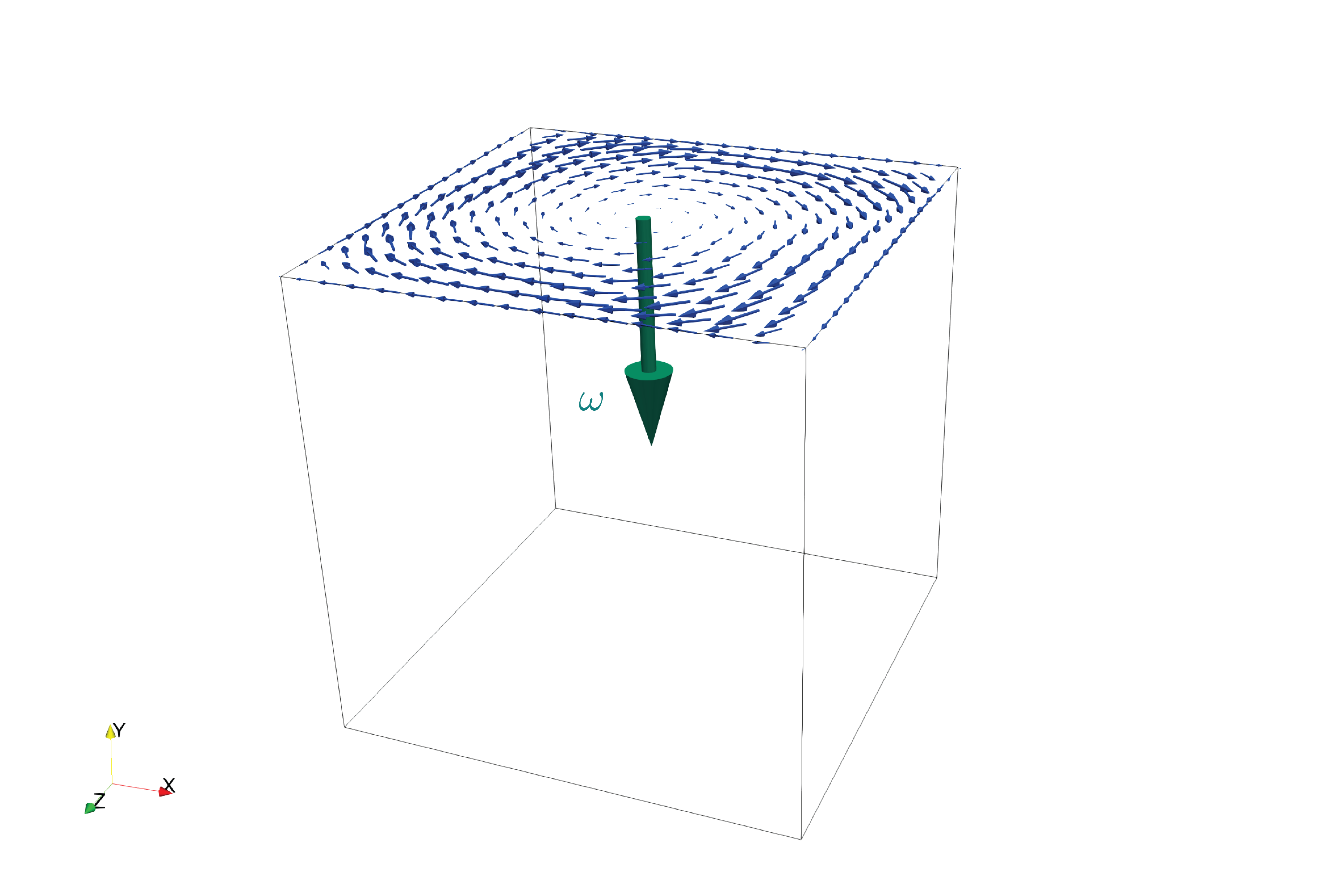}
  \caption{Driven cavity with rigid body rotation $\omega$ in one of its faces.}
  \label{fig:dc_bc_noheat}
\end{figure}

\begin{figure}
  \centering
  \begin{tabular}{c c}
  \includegraphics[width = 0.45\textwidth]{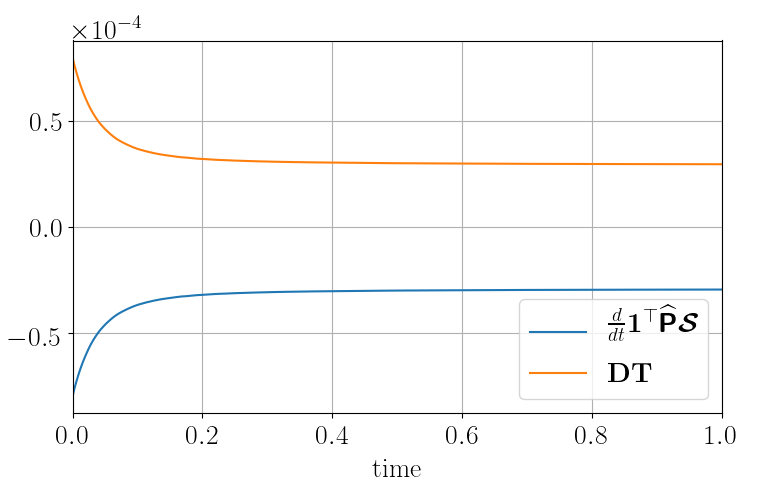} & \includegraphics[width = 0.45\textwidth]{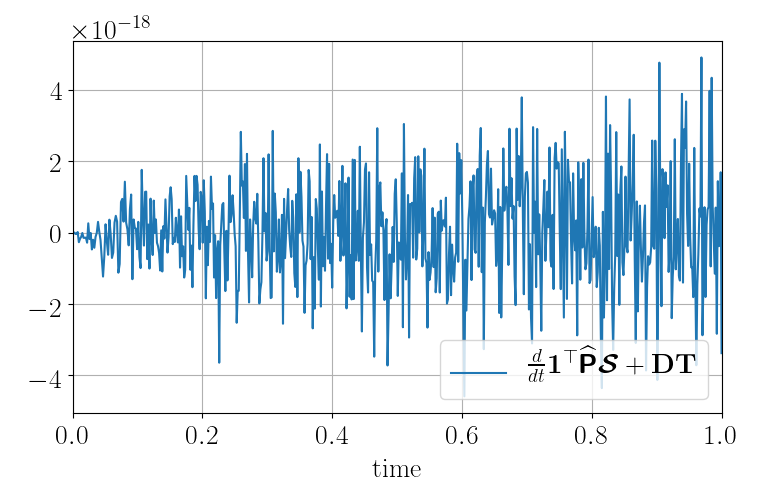}\\
  \end{tabular}
	\caption{\reftwo{Lid-driven cavity. Left: discrete integral of the time rate of change of the entropy function, $\frac{d}{dt} \mathbf{1}^{\top} \widehat{\Pmatvol} \, \bm{\fnc{S}}$\dadd{,} and discrete dissipation term\dadd{,} $\mathbf{DT}$ (see equation \eqref{eq:estimate-no-slip-bc-2}). Right: instantaneous entropy balance (see equation \eqref{eq:estimate-no-slip-bc-2}).}}
  \label{fig:ldc_dSdt_error}
\end{figure}
%-- old
%We first report on the results of a simple test designed to validate the entropy conservation properties of the
%interior domain SBP-SAT algorithm implemented in our code,
%and of the newly developed solid wall boundary conditions. As a test-bed, we used the three-dimensional
%lid-driven cavity with adiabatic solid walls\dadd{.} 
\dadd{Next, we validate the algorithm on the simple problem of the three-dimensional
lid-driven cavity with adiabatic solid walls. The domain is} a cube \dadd{with sides of length} $l$ discretized using a
Cartesian grid composed of eight elements in each direction.
A velocity field is imposed on one of the walls,
corresponding to a rigid body rotation about the
center of the wall at a speed $\omega$ (see figure \ref{fig:dc_bc_noheat})\dadd{.} \dadd{Based} on the rotation velocity and the length of the cavity,
this example is characterized by
a Reynolds number $Re=l^2\omega/\nu=100$ and a Mach number $M=l \omega/c=0.05$. All the dissipation terms used for the interface coupling \cite{parsani_entropy_stable_interfaces_2015} and the imposition of the boundary
conditions are
turned off, including upwind and interior-penalty SAT terms. \reftwo{The two terms \dadd{on} the left-hand side of equation \eqref{eq:estimate-no-slip-bc-2}, $\frac{d}{dt} \mathbf{1}^{\top} \widehat{\Pmatvol} \, \bm{\fnc{S}} + \mathbf{DT}$, are
monitored at every time step. Note that because we do not include any dissipation terms for the interface couplings and adiabatic wall boundary conditions are used, $\mathbf{\Xi}$ in equation \eqref{eq:estimate-no-slip-bc-2} is zero.}
The monitored values are reported in Figure \ref{fig:ldc_dSdt_error}, left panel,
together with the error committed in entropy conservation, right panel, which is below machine (double) precision.
Therefore, we have confirmed numerically that the newly developed wall boundary 
conditions together with the interior discretization operator are entropy conservative 
when all the entropy dissipative terms are turned off.

We also present the results of a modified version of the lid-driven cavity example by considering \dadd{non-adiabatic} boundaries, and by imposing a nonzero entropy flux in a face adjacent to the rotating face\dadd{, given} as $\mathtt{g}(t)=10^{-4}\sin(4\pi t)$ (see figure \ref{fig:dc_bcs}). Figure \ref{fig:ldc_dSdt_g_t_error} shows the terms on the entropy balance for this example, including the boundary contribution (left panel), and the error in the entropy conservation (right), which is below machine \reftwo{(double)} precision.

\begin{figure}
  \centering
  \includegraphics[width=0.6\textwidth]{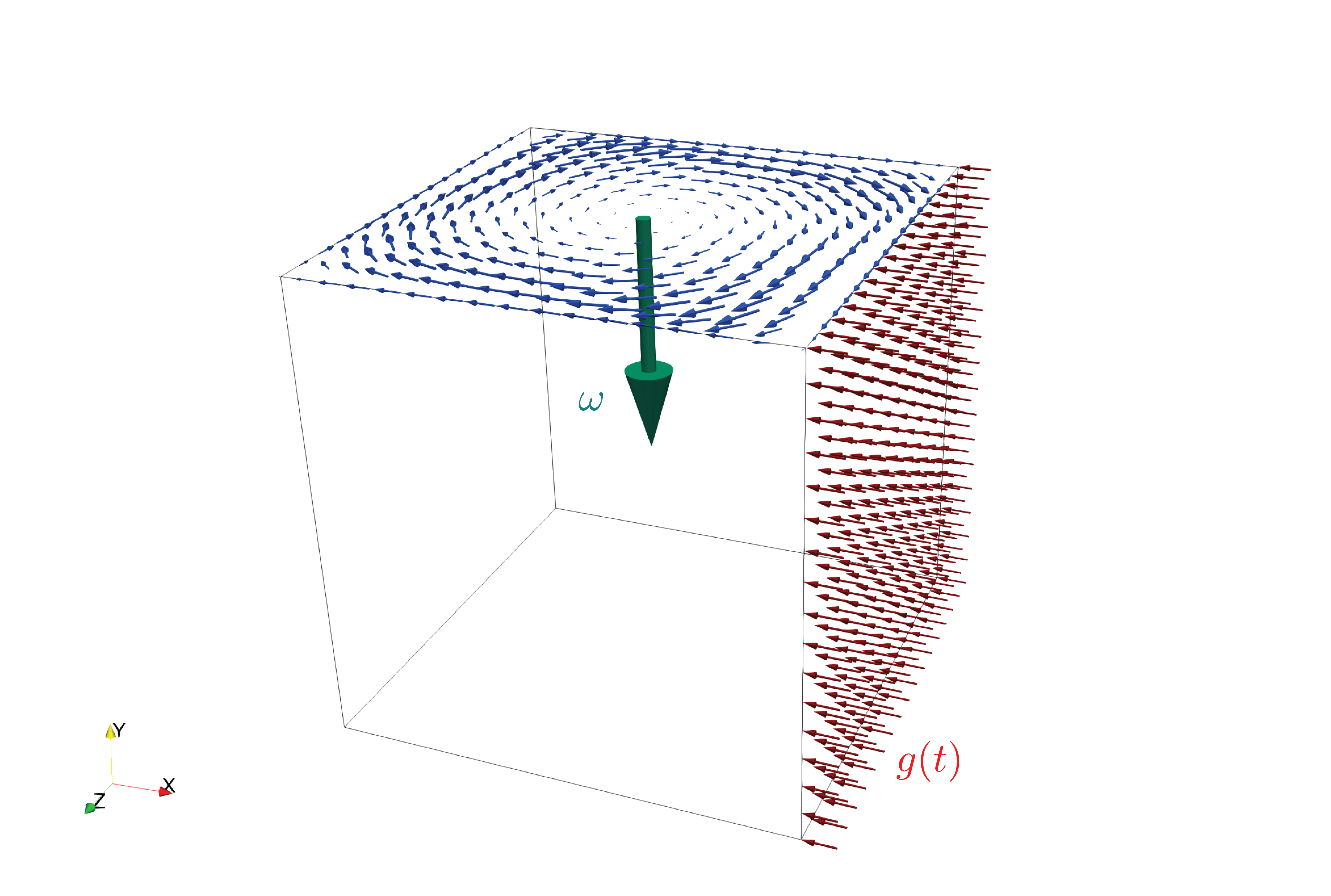}
  \caption{Driven cavity with rigid body rotation $\omega$ in one of its faces and a time-dependent entropy flux $\mathtt{g}(t)$ on and adjacent face.}
  \label{fig:dc_bcs}
\end{figure}

\begin{figure}
  \centering
  \begin{tabular}{c c}
  \includegraphics[width = 0.45\textwidth]{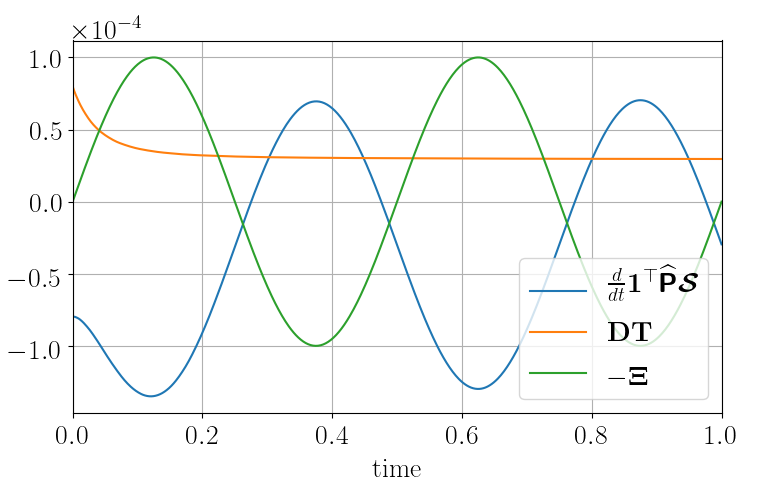} & \includegraphics[width = 0.45\textwidth]{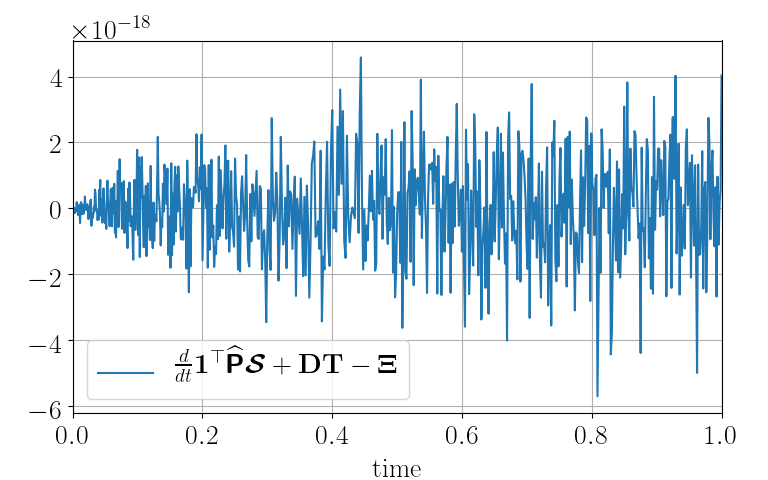}\\
  \end{tabular}
	\caption{\reftwo{Lid-driven cavity with nonzero entropy flux. Left: discrete integral of the time rate of change of the entropy function, $\frac{d}{dt} \mathbf{1}^{\top} \widehat{\Pmatvol} \, \bm{\fnc{S}}$, discrete dissipation term\dadd{,} $\mathbf{DT}$\dadd{,} and discrete integral of the boundary data contribution, $\mathbf{\Xi} = \Pxgentwodhat{2}{3} \, \mathtt{g}(t)$ (see equation \eqref{eq:estimate-no-slip-bc-2}). Right: instantaneous entropy balance (see equation \eqref{eq:estimate-no-slip-bc-2}).}}
  \label{fig:ldc_dSdt_g_t_error}
\end{figure}

\subsection{Subsonic flow past a cylinder}

We further explore the engineering capabilities of the entropy\dadd{-}stable SBP-SAT operators and the \dadd{numerical procedure for the weak imposition of the} solid wall boundary
conditions by simulating the flow around a cylinder, a canonical example of external flows with
important applications such as particle transport, fluid-structure interaction and bluff body
aerodynamics, that has been extensively studied both numerically
\cite{karniadakis_triantafyllou_1992,barkley_henderson_1996,henderson_1997,Park1998}
 and experimentally \cite{williamson_1989,norberg_1994,zhang_1995,williamson_1996,prasad_williamson_1997}.

We described the flow in a Cartesian coordinate system ($x_1$,$x_2$,$x_3$), with
the free-stream velocity aligned in the $x_1$ direction. A circle of diameter
$d$ is centered at the origin, with the domain of interest delimited by
a rectangular box that respectively extends $20d$ and $60d$ upstream and downstream\dadd{ of} the 
flow direction, and $30d$ in the $x_2$ direction.
Such a 2D domain is then extruded a distance $d$ in the $x_3$ direction. We prescribe \dadd{an} 
adiabatic no-slip wall boundary condition on the surface of the cylinder; periodic boundaries are \dadd{applied} in the $x_3$ direction.
The remaining faces of the box are treated as a far field.

The domain \reftwo{is} discretized as follows: we first mesh the $(x_1,x_2)$ plane with second-order quadrilateral elements, and include a boundary layer around the cylinder wall
 as pictured in Figure \ref{fig:cylindermesh}. We then extrude this mesh using \reftwo{three} layers of elements in the $x_3$ direction over the span of the cylinder.

\begin{figure}
  \centering
  \includegraphics[width=0.6\textwidth]{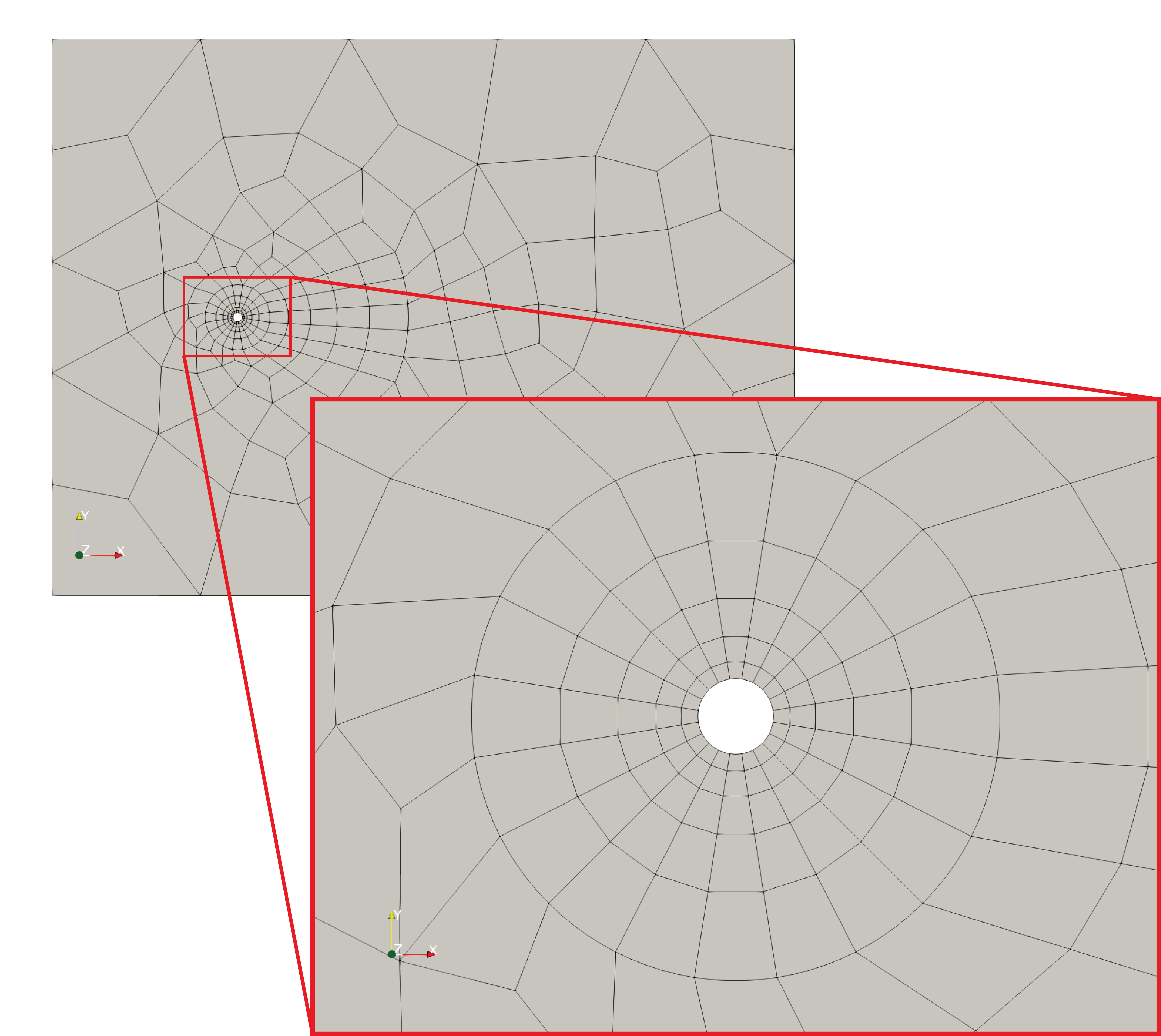}
  \caption{Subsonic flow past a cylinder: mesh cut on the $x_3$ plane, with a zoom of the boundary layer area.}
  \label{fig:cylindermesh}
\end{figure}

Considering the free-stream velocity $u_\infty$
and the diameter of the cylinder $d$, the free-stream flow is characterized by a Mach number $M_\infty=0.1$ and a
Reynolds number $Re_\infty=300$. Under these conditions, the flow developed
behind the cylinder is three-dimensional and \dadd{results in} vortex shedding at a constant
frequency.

In Table \ref{tab:cylindergrid_all} we report the time average drag coefficient\dadd{,}
$\bar{c}_d$\dadd{,} and the Strouhal number\dadd{,} $St=fd/u_\infty$ based on the
frequency of the vortex shedding, \dadd{resulting from} uniform $p$-refinements (with polynomial orders for the solution space ranging from $p=2$ to $p=5$)
and three levels of uniform $h$-refinement of the initial mesh described below, resulting in 714 (denoted by $rl=0$ in Table \ref{tab:cylindergrid_all}), 5,712 ($rl=1$), and 45,969 ($rl=2$) hexahedral elements respectively.
Comparisons with results reported in the literature for these aerodynamic coefficients \cite{Henderson95} are also provided.
From these tables, it can be seen that in all cases the accuracy of the results improve by increasing the order of accuracy of the scheme and the grid resolution. 
\dadd{The} fifth- ($p=4$) and sixth-order ($p=5$)
accurate entropy\dadd{-}stable schemes perform very well on the last level of $h-$refinement considered ($rl=2$), which is coarser compared to the typical grids used with second-order finite volume and finite differences schemes.

% Tables with the numerical results for the flow around a cylinder at Re=300
\begin{table}[]
  \centering
  \begin{tabular}{||c||c|c|c||}
    \hline \hline
     Poly. order & $rl=0$ & $rl=1$ & $rl=2$\\
    \hline
    \begin{tabular}{c}
                \\
      $p=2$     \\
      $p=3$     \\
      $p=4$     \\
      $p=5$     \\
      \cite{Henderson95}
    \end{tabular} &
    \begin{tabular}{cc}
      $\bar{c}_d$ & $St$ \\
      1.579     & 0.1995 \\
      1.439     & 0.1997 \\
      1.458     & 0.1998 \\
      1.414     & 0.1999 \\
      1.381     & 0.2000
    \end{tabular} &
    \begin{tabular}{cc}
      $\bar{c}_d$ & $St$ \\
      1.541     & 0.1998 \\
      1.439     & 0.1999 \\
      1.386     & 0.1999 \\
      1.382     & 0.1999 \\
      1.381     & 0.2000
    \end{tabular} &
    \begin{tabular}{cc}
      $\bar{c}_d$ & $St$ \\
      1.443     & 0.1999 \\
      1.383     & 0.1999 \\
      1.384     & 0.1999 \\
      1.386     & 0.1999 \\
      1.381     & 0.2000
    \end{tabular}\\
    \hline
  \end{tabular}
  \caption{Subsonic flow past a cylinder: mean drag coefficient and Strouhal number for different polynomial orders $p$ and uniform refinements $rl$, $Re_{\infty}=300$, $M_{\infty}=0.1$. Reference values are also provided.}
  \label{tab:cylindergrid_all}
\end{table}

\subsection{Subsonic flow past a sphere}

We then test our implementation within a more complicated setting represented by the flow around a sphere.
%The flow around a sphere is a similar problem to the mentioned flow around a
%cylinder that has also been extensively studied.
%Additional complications arise
%from the nature of the 3D domain, the required discretization and from the flow
%regime.
%A similar domain to the one used for the flow around a cylinder
%is constructed for this example.
In this case, a sphere of diameter $d$ is centered
at the origin, and a box is respectively extended $20d$ and
$60d$ upstream and downstream \dadd{of the flow direction}; the box size is $30d$ in both the $x_2$ and $x_3$ directions.
As boundary conditions, we consider adiabatic solid walls at the surface of the sphere and far field on all faces of
the box.

The surface of the sphere is first triangulated
using second-order simplices and a boundary layer composed of triangular prisms is
extruded from the sphere surface for a total length of $3d$. The rest of the domain is meshed with an
unstructured tetrahedral mesh. We then obtain an unstructured conforming hexahedral mesh by uniformly splitting each tetrahedron \dadd{in to} \reftwo{four} hexahedral elements,
and each prism \dadd{in to} \reftwo{three} hexahedral elements, resulting in a total of 4,328 hexahedral elements. A cut of the final mesh is illustrated in Figure \ref{fig:sphere_mesh},
together with a representative splitting of a tetrahedral and a prismatic cell.

\begin{figure}[t!]
  %\centering
  \begin{subfigure}[t]{0.6\textwidth}
    \centering
    \includegraphics[width=\textwidth]{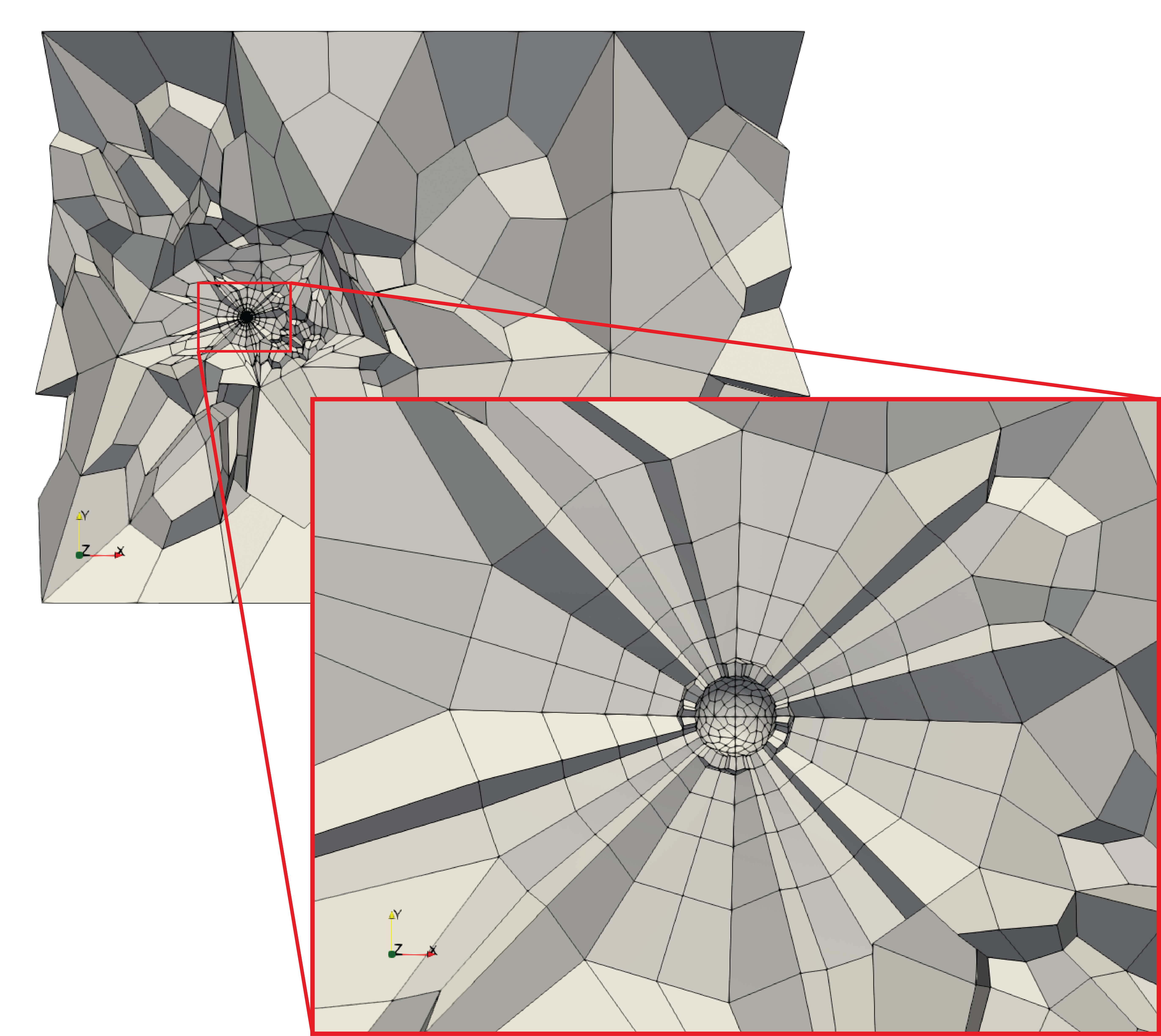}
  \end{subfigure}
  \begin{subfigure}[t]{0.4\textwidth}
    \centering
    \includegraphics[width=\textwidth,height=1.3\textwidth]{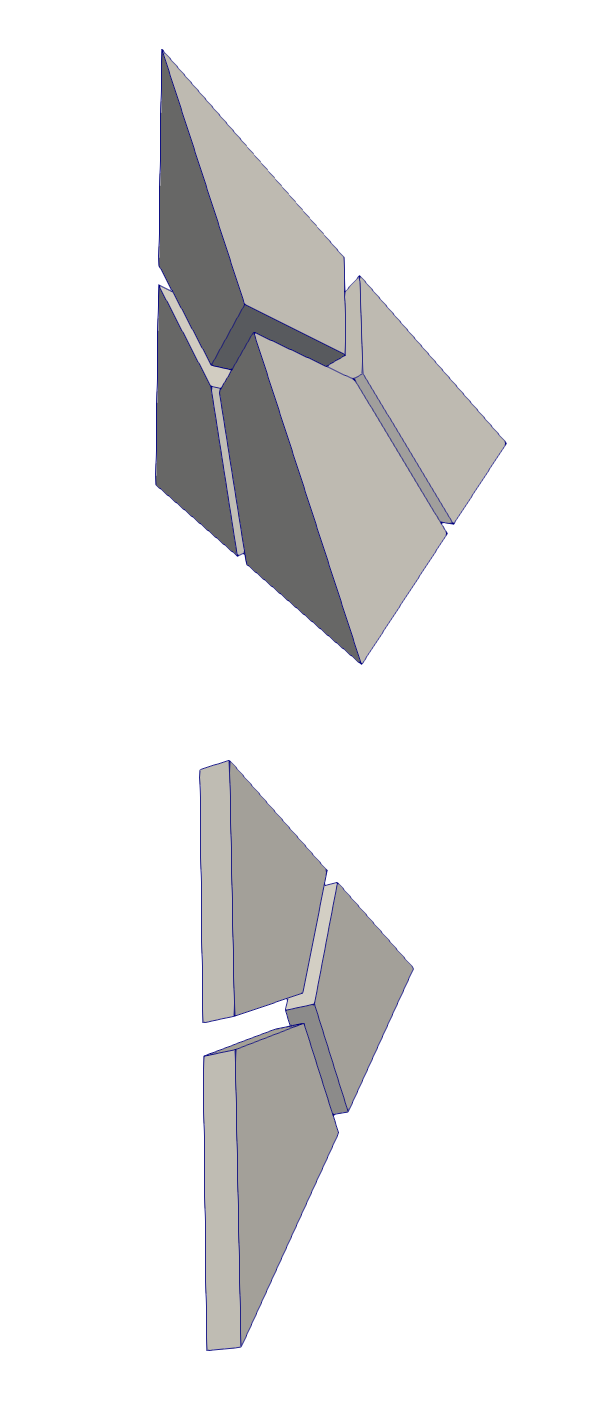}
  \end{subfigure}
  \caption{Subsonic flow past a sphere. Left: mesh cut on the $x_3$ plane, with a zoom of the boundary layer area. Right panel, top: tetrahedral to hexahedral refinement. Right panel, bottom: prismatic to hexahedral refinement.}
  \label{fig:sphere_mesh}
\end{figure}

The free-stream flow is characterized by a Mach number $M_\infty=0.1$ and a
Reynolds number $Re_\infty=300$. Under these conditions, the flow developed behind the
sphere is non-axisymmetric, with hairpin vortices shedding from the wake at a constant rate \cite{MITTAL20084825}, inducing
a total non-zero lift force on the sphere. Figure \ref{fig:sphere_lic} \dadd{depicts} these hairpin vortices with isocontours of the Q-criterion colored by the vorticity magnitude at a given time instant.
%From this plot, it can be clearly seen that
%the laminar flow past the object, the sphere's wake and the shedding of vortices
\begin{figure}
  %\centering
  %\begin{subfigure}[t]{0.5\textwidth}
  %  \centering
  %  \includegraphics[width=\textwidth]{figures/sphere_lic_temperature.png}
  %\end{subfigure}
  %\begin{subfigure}[t]{0.5\textwidth}
    \centering
    \includegraphics[width=0.6\textwidth]{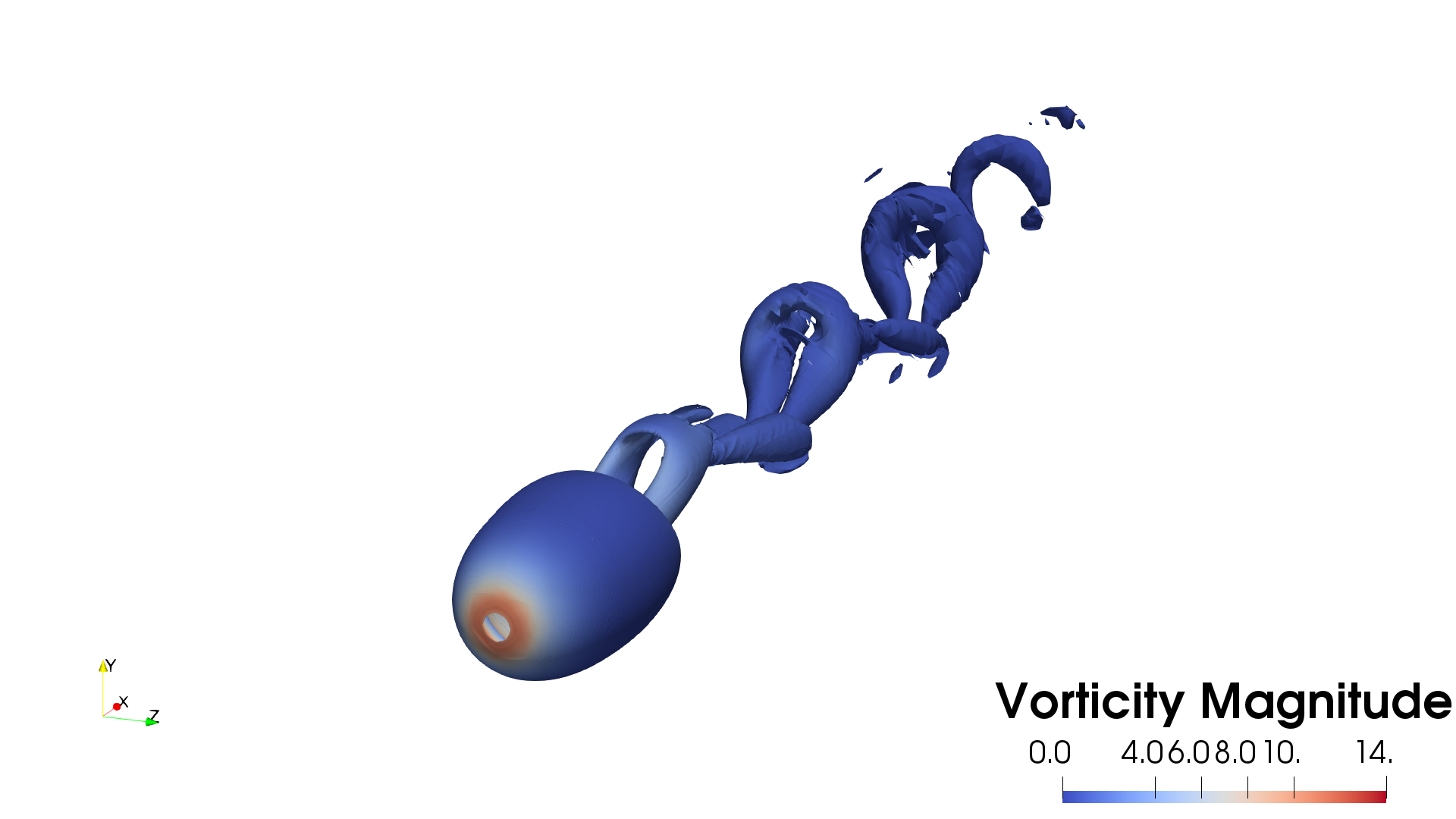}
  %\end{subfigure}
    %\caption{Subsonic flow past a sphere: line integral convolution for the velocity field colored by contours of temperature at $t=150$, $Re_{\infty}=300$, $M_{\infty}=0.1$. Solution with $l=2$ and $p=4$.}
    \caption{Subsonic flow past a sphere: isocontour 0.02 of the Q-criterion colored by vorticity magnitude at $t=150$, $Re_{\infty}=300$, $M_{\infty}=0.1$. Solution with $rl=2$ and $p=5$. Notice the hairpin vortices with a nearly constant orientation.}
  \label{fig:sphere_lic}
\end{figure}

As in the previous test case, we analyze the convergence of some quantities
of interest under $p$- and $h$-refinement. The number of elements in the sequence of nested grids considered are 4,328, 34,624 and 276,992 respectively; the polynomial order of the solution ranges from $p=2$ to $p=5$.
For this test, we monitor the time average drag coefficient of the sphere\dadd{,}
$\bar{c}_d$, the average lift coefficient\dadd{,} $\bar{c}_l$ (with its orientation $\theta$ in the $x_3-x_2$ plane),
and the Strouhal number $St$ based on the frequency of the vortex shedding.
We remark that $\bar{c}_l$ is nonzero because of the asymmetric flow features \cite{MITTAL20084825}.
The results provided in Table \ref{tab:spheregrid_all} are in good agreement with those reported in the literature \cite{johnson_patel_1999} for sufficiently refined meshes (i.e. $rl=1$ and $rl=2$),
and polynomial orders greater \dadd{than} or equal to 3.

%Figure \ref{fig:sphere_lic} shows the
%velocity field at $t=150$ for the solution of the flow around a sphere with
%$l=2$, and $p=4$ using the line integral convolution method. At
%this time the flow has reached a constant periodic state. From this plot, it can
%be seen the asymmetric view of the flow with a preferred direction almost
%aligned with the $x_2$ axis, in agreement with our reported value of $\theta$
%for this solution.

Figure \ref{fig:sphere_lic} shows the 0.02 isocontour of the Q-criterion at $t=150$ for the solution of the flow around a sphere with
$rl=2$, and $p=5$. At this time\dadd{,} the flow has reached a constant periodic state. From this figure, it can
be seen the asymmetric pattern of the flow, consistent with a nonzero $\bar{c}_l$. Notice the heads of the hairpin vortices are
almost aligned with the $x_2$ axis, in agreement with our reported value of $\theta$
for this solution.

% Tables with the numerical results for the flow around a sphere at Re=300
\begin{table}[]
  \centering
  \begin{footnotesize}
  \begin{tabular}{||c||c|c|c||}
    \hline \hline
     Poly. order& $rl=0$ & $rl=1$ & $rl=2$\\
    \hline
    \begin{tabular}{c}
                \\
      $p=2$     \\
      $p=3$     \\
      $p=4$     \\
      $p=5$     \\
      \cite{johnson_patel_1999}
    \end{tabular} &
    \begin{tabular}{cccc}
      $\bar{c}_d$ & $\bar{c}_l$ & $\theta$ & $St$ \\
      1.125 & 0.095 & -0.34   & 0.060 \\
      0.815 & 0.074 & -0.18   & 0.120 \\
      0.688 & 0.069 & 0.14    & 0.140 \\
      0.650 & 0.066 & 0.39    & 0.140 \\
      0.656 & 0.069 & -       & 0.137
    \end{tabular} &
    \begin{tabular}{cccc}
      $\bar{c}_d$ & $\bar{c}_l$ & $\theta$ & $St$ \\
      0.843 & 0.076 & -0.10    & 0.140 \\
      0.671 & 0.068 &  0.24    & 0.140 \\
      0.656 & 0.067 &  1.80    & 0.140 \\
      0.659 & 0.067 &  1.68    & 0.140 \\
      0.656 & 0.069 & -       & 0.137
    \end{tabular} &
    \begin{tabular}{cccc}
      $\bar{c}_d$ & $\bar{c}_l$ & $\theta$ & $St$ \\
      0.691 & 0.069 &  0.45    & 0.140 \\
      0.659 & 0.068 &  1.71    & 0.140 \\
      0.659 & 0.068 &  1.70    & 0.140 \\
      0.659 & 0.068 &  1.70    & 0.140 \\
      0.656 & 0.069 & -       & 0.137
    \end{tabular}\\
    \hline
  \end{tabular}
  \caption{Subsonic flow past a sphere: mean drag coefficient $\bar{c}_d$, mean lift coefficient $\bar{c}_l$ and its orientation in the $x_1=0$ plane $\theta$, and Strouhal number $St$ for different polynomial orders $p$ and uniform refinements $l$, $Re_{\infty}=300$, $M_{\infty}=0.1$. Reference values are also provided.}
  \label{tab:spheregrid_all}
  \end{footnotesize}
\end{table}

\subsection{Supersonic flow past a square cylinder}

We finally provide further evidence of the robustness of the algorithm  in the context of supersonic flow around a square cylinder with $Re_\infty=10^4$ and
$M_\infty=1.5$, which features shocks, expansion regions and three-dimensional vortical structures \cite{parsani_entropy_stability_solid_wall_2015}.
We start with a square of side $s$ placed in the $x_1-x_2$
plane, and construct an unstructured mesh around it, manually refined
in order to capture the main features of the flow (see Figure
\ref{fig:supersonicmesh}). We then extrude the mesh for a total size $s$ in the $x_3$ direction using four elements; the final three-dimensional mesh used in the study consists of  87,872 hexahedral elements.
The boundary conditions imposed are adiabatic solid wall on the square cylinder surfaces, periodic boundary
conditions in the $x_3$ direction\dadd{,} and far field \dadd{at} the remaining boundaries. The problem is solved using a fourth-order accurate ($p=3$) discretization.

\begin{figure}
  \centering
  \includegraphics[width=0.6\textwidth]{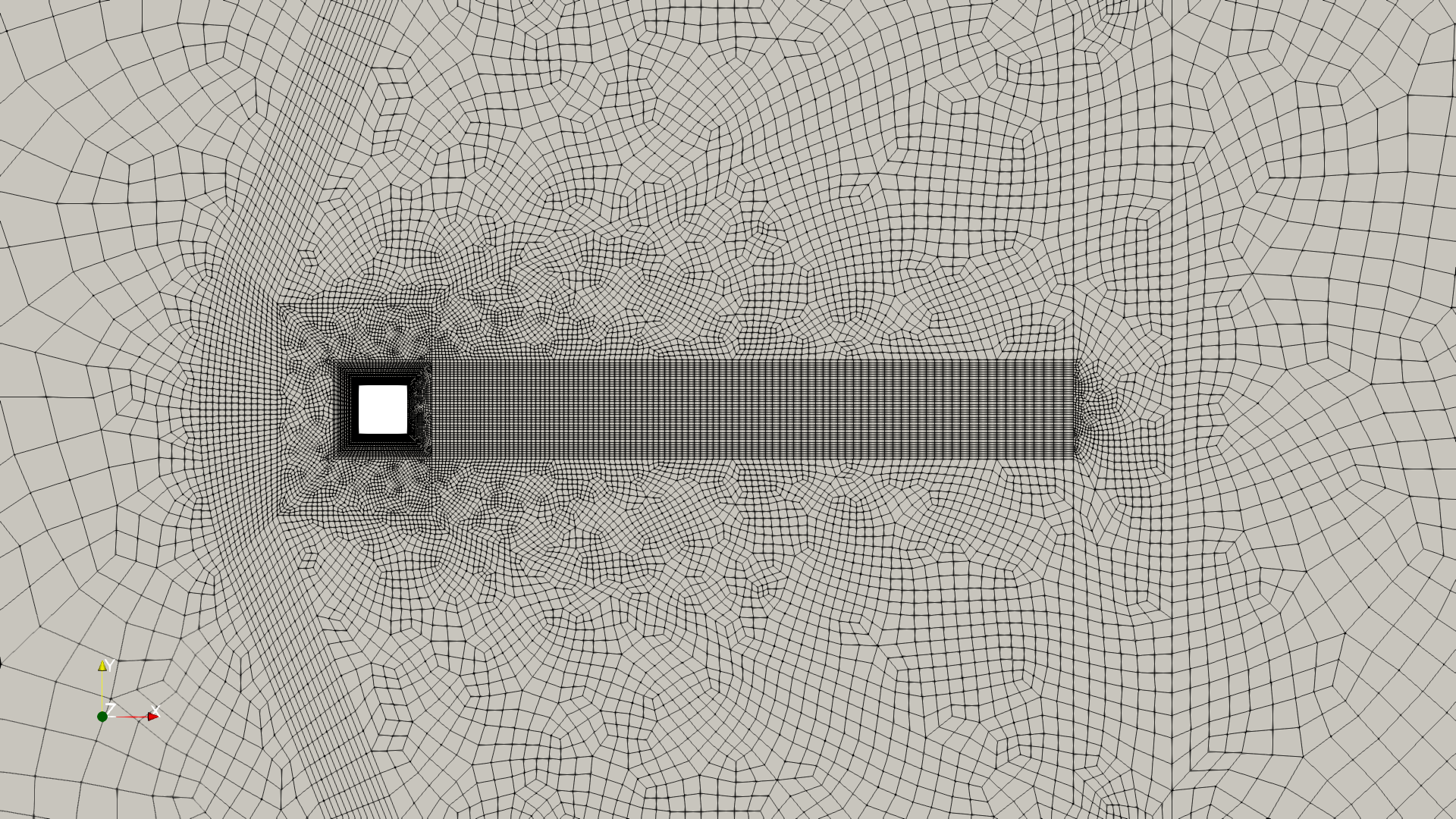}
  \caption{Mesh used for the solution of the supersonic flow past a square cylinder.}
  \label{fig:supersonicmesh}
\end{figure}

\begin{figure}
  \centering
  \begin{tabular}{c c}
    \includegraphics[width = 0.45\textwidth]{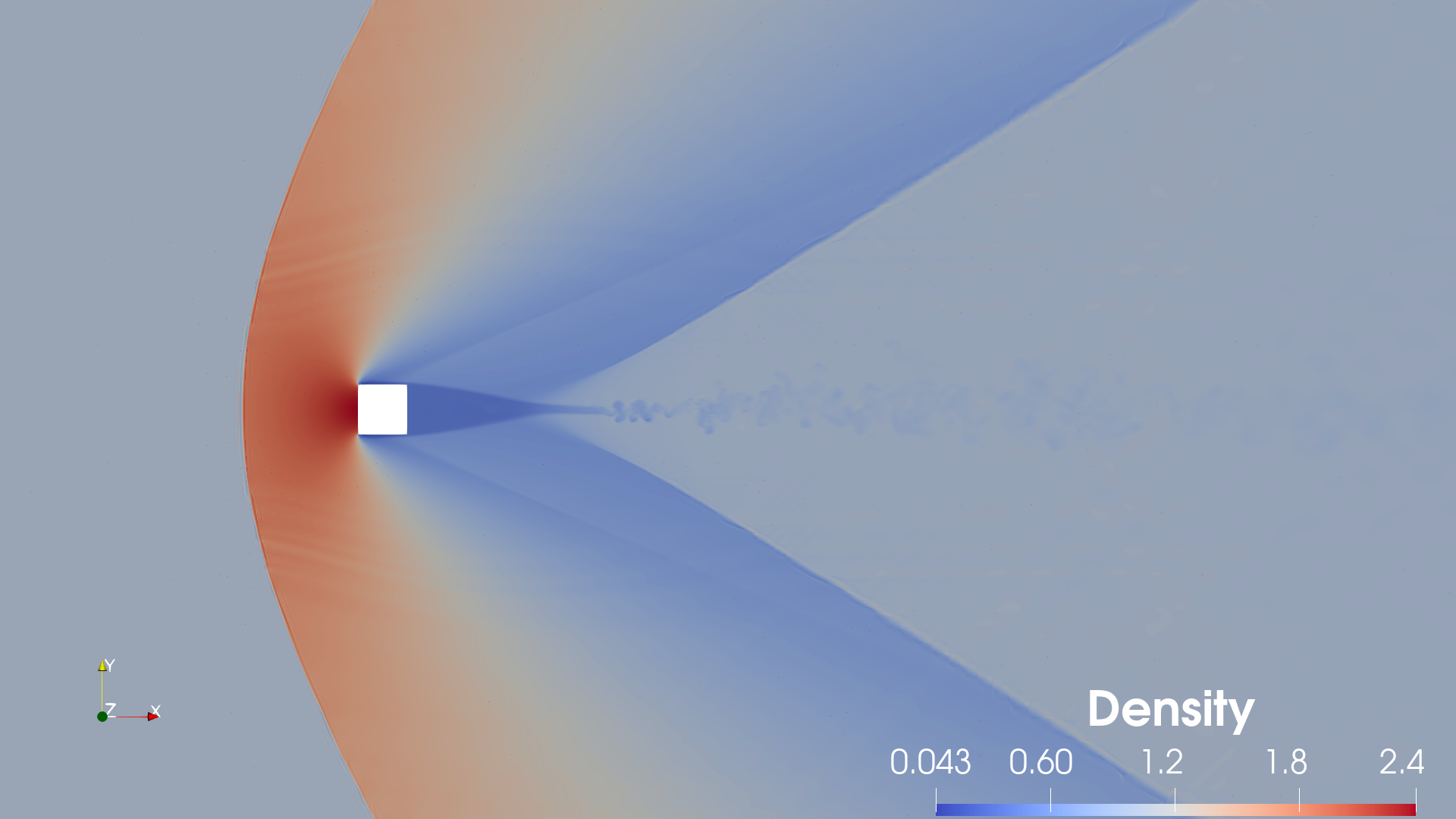} & \includegraphics[width = 0.45\textwidth]{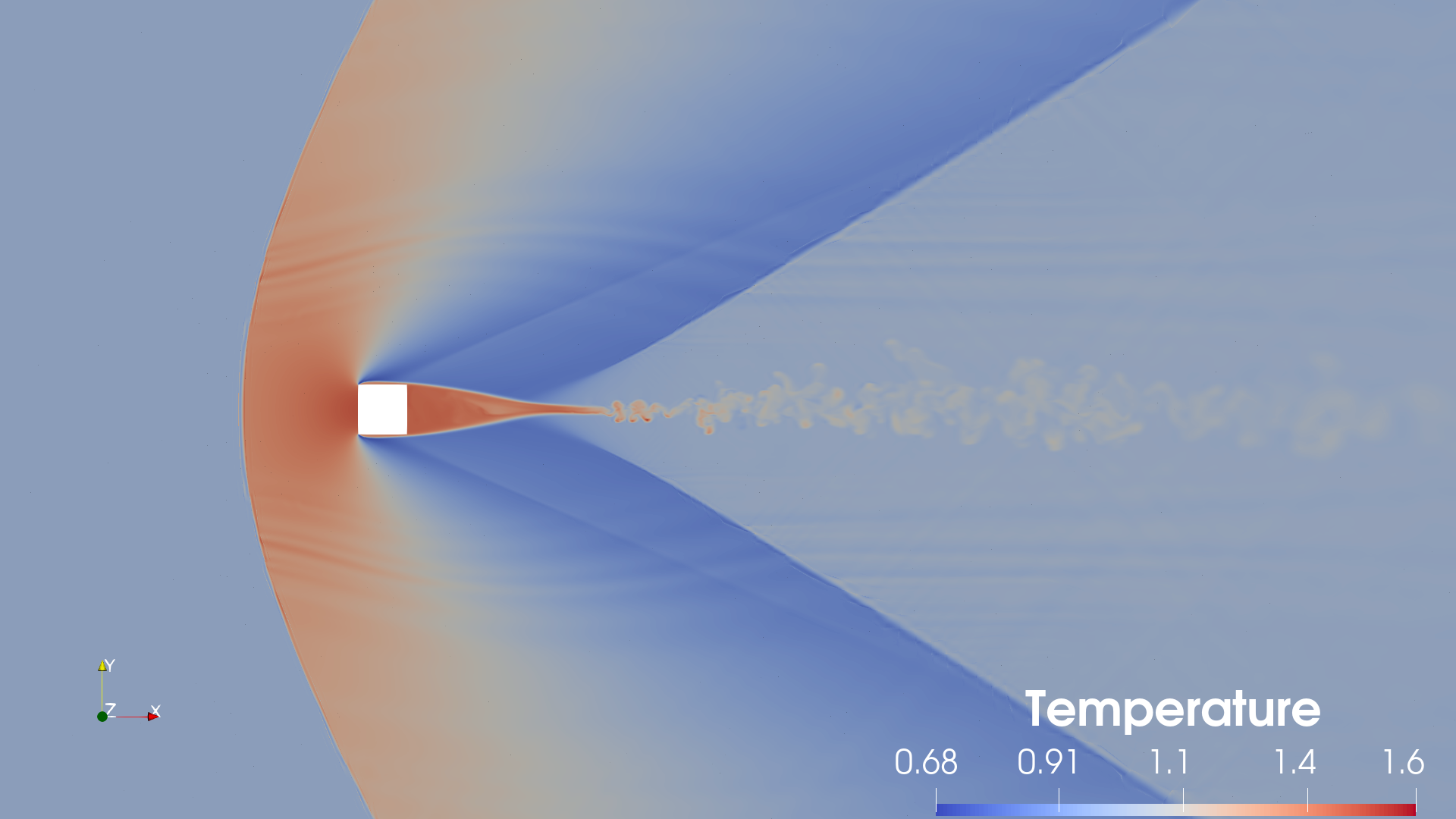}\\
    \includegraphics[width = 0.45\textwidth]{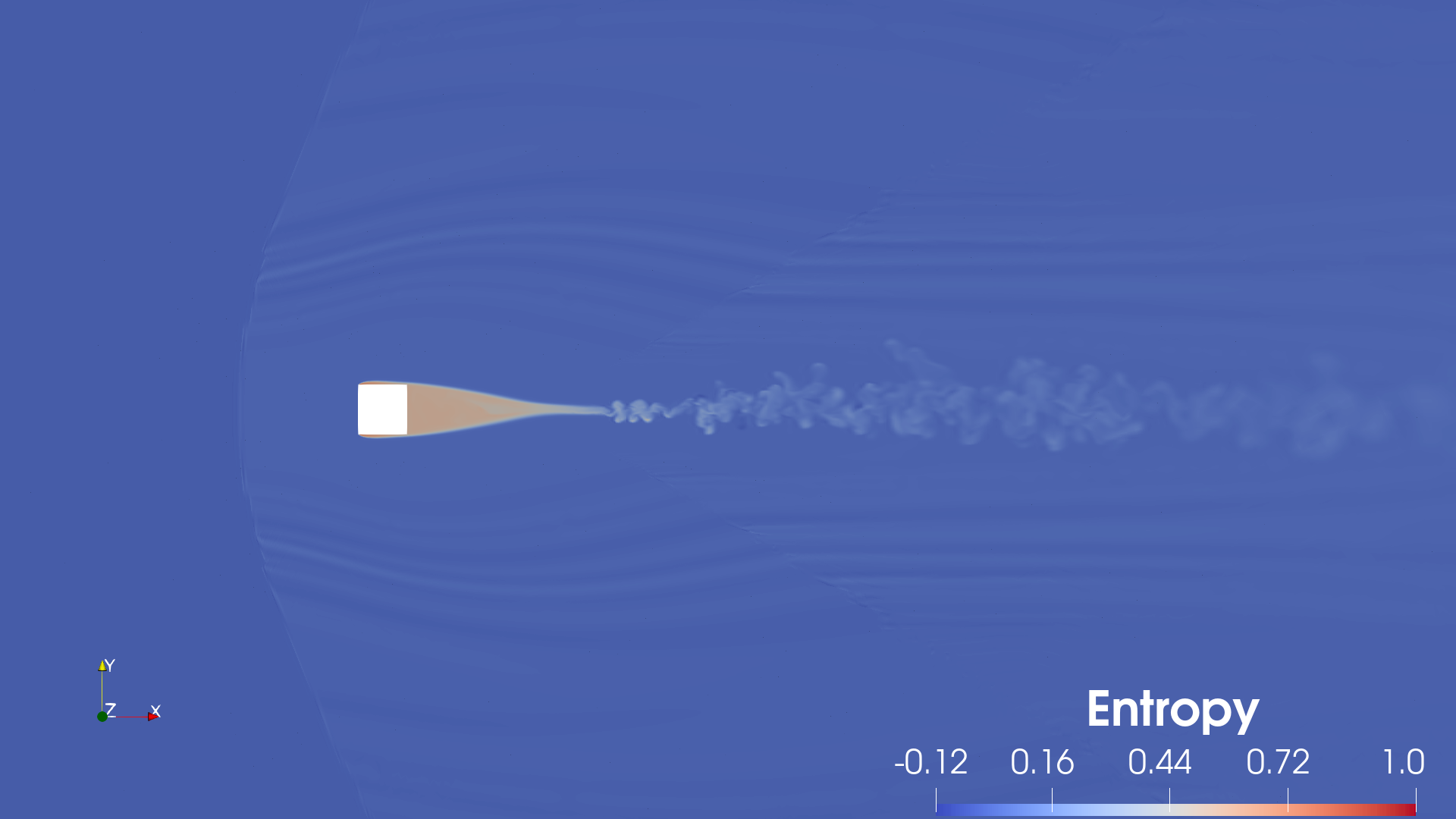} & \includegraphics[width = 0.45\textwidth]{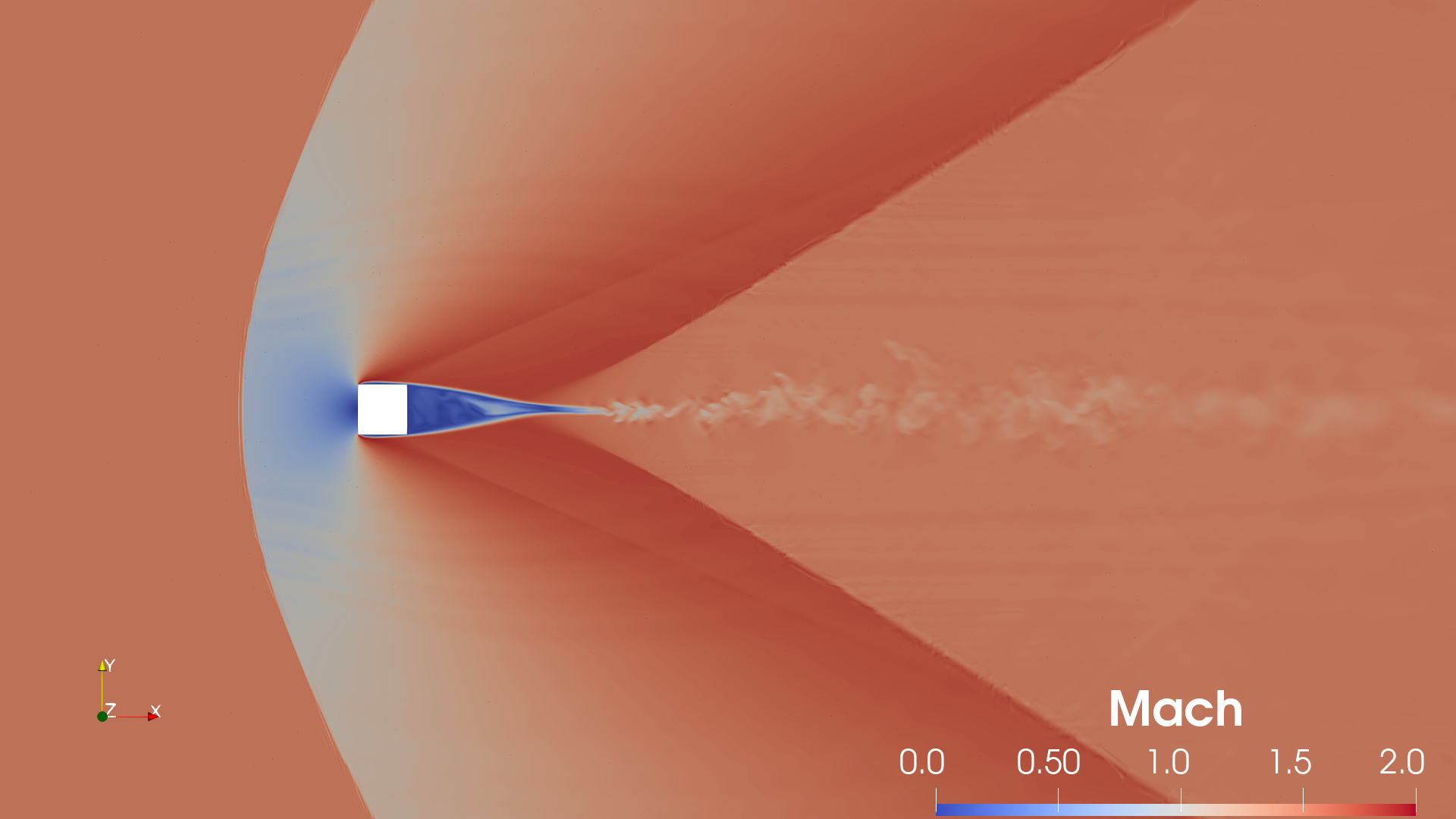}\\
  \end{tabular}
  \caption{Density, temperature, entropy and local Mach number contours for the supersonic flow around a square cylinder with $Re_\infty=10^4$ and $M_\infty=1.5$ at $t=100$.}
  \label{fig:ssresults}
\end{figure}

%% \begin{figure}
%%   \centering
%%   \includegraphics[width=0.6\textwidth]{figures/ssdensity.png}
%%   \caption{Density contours for the supersonic flow around a square cylinder with $Re_\infty=10^4$ and $M_\infty=1.5$ at $t=100$.}
%%   \label{fig:ssdensity}
%% \end{figure}

%% \begin{figure}
%%   \centering
%%   \includegraphics[width=0.6\textwidth]{figures/sstemperature.png}
%%   \caption{Temperature contours for the supersonic flow around a square cylinder with $Re_\infty=10^4$ and $M_\infty=1.5$ at $t=100$.}
%%   \label{fig:sstemperature}
%% \end{figure}

%% \begin{figure}
%%   \centering
%%   \includegraphics[width=0.6\textwidth]{figures/ssentropy.png}
%%   \caption{Entropy contours for the supersonic flow around a square cylinder with $Re_\infty=10^4$ and $M_\infty=1.5$ at $t=100$.}
%%   \label{fig:ssentropy}
%% \end{figure}

%% \begin{figure}
%%   \centering
%%   \includegraphics[width=0.6\textwidth]{figures/ssMach.png}
%%   \caption{Local Mach number contours for the supersonic flow around a square cylinder with $Re_\infty=10^4$ and $M_\infty=1.5$ at $t=100$.}
%%   \label{fig:ssMach}
%% \end{figure}

Figure \ref{fig:ssresults} show the results for the
supersonic square cylinder at $t=100$. At this point in time, the flow is fully unsteady
and the shock in front of the cylinder has reached its final position. The flow is characterized by the shock in front of the square cylinder and
those in the near wake region. There is also an unsteady wake populated by three-dimensional vortices
shedding from the body.

We finally remark that the small oscillations near the shock
region are caused by discontinuities in the solution and are expected for this scheme.
In fact, we are not using any shock capturing method or reducing the order of
scheme at the discontinuity. Nevertheless, the simulation remains stable at all
time, and the oscillations are always confined \dadd{to} small regions \dadd{near} the
discontinuities. \reftwo{This is a feat unattainable with several alternative approaches 
to wall boundary conditions based on linear analysis which for this test problem lead to
numerical instabilities and an almost immediate crash of the solver.}

%%%%%%%%%%%%%%%%%%%%%%%%%%%%%%%%%%%%%%%%%%%%%%%%%%%%%%%%%%%%%%%%%%%%%%%%%%%%%%%%
\section{Conclusions}\label{sec:conclusions}

We have used entropy stability and \dadd{the} summation-by-parts framework to derive entropy stable wall boundary conditions
for the three-dimensional compressible Navier--Stokes equations in the presence of an adiabatic wall, or a wall with a prescribed heat entropy flow.

A point-wise entropy\dadd{-}stable numerical procedure has been presented for weakly enforcing
these solid wall boundary conditions at the semi-discrete level combining a penalty
flux approach with a simultaneous-approximation-term technique for the
conservative variables and the variables representing the gradient of entropy.
The resulting semi-discrete operator mimics exactly the behavior at the continuous level, and the proposed non-linear boundary treatment provides a mechanism for ensuring non-linear stability in
the $L^2$ norm of the continuous and semi-discretized compressible Navier--Stokes equations.

\dadd{The design order properties of the scheme are validated in the context of 
laminar flow in a pipe with an annular section.} Detailed viscous numerical computations in a three-dimensional subsonic
lid-driven cavity flow have been presented
to assess the accuracy of the proposed numerical techniques. The
error in the entropy function balance showed an excellent agreement with the
theory with or without heat entropy flux.

Unsteady laminar flow past a cylinder and a sphere have been presented to highlight
the efficacy in computing aerodynamic forces; numerical simulations considering both $p$- and $h$-refinements
showed very good agreement with results available from the literature.

The robustness of the complete semi-discrete operator
(i.e. the entropy\dadd{-}stable interior operator coupled with the new boundary
treatment) \dadd{was} demonstrated for the supersonic flow past a
three-dimensional square
cylinder at $Re_{\infty}=10\time10^4$ and $M_{\infty} = 1.5$, as proposed in
\cite{parsani_entropy_stability_solid_wall_2015}. This test has
been successfully computed with a fourth-order accurate method without the
need \dadd{of introducing} artificial dissipation, limiting techniques, or filtering,
for the purpose of stabilizing the computations, a feat unattainable with several alternative
approaches based on linear analysis only.

Although the robustness and efficacy of the techniques presented in this work have been validated
using discontinuous spectral collocation operators on unstructured grids,
the new boundary conditions can be applied to a very broad class of
spatial discretizations and they are compatible with any diagonal-norm summation-by-parts spatial operator,
including finite element, finite difference, finite volume, discontinuous Galerkin, and
flux reconstruction schemes.

\section*{Acknowledgments}
The research reported in this paper was funded by King Abdullah University of
Science and Technology. We are thankful for the computing resources of the
Supercomputing Laboratory and the Extreme Computing Research Center at
King Abdullah University of Science and Technology.

\section*{References}
\bibliographystyle{aiaa}
\bibliography{bc}

\begin{appendix}

\section{Python script for the verification of the proofs in three dimension using 
  curvilinear grids}\label{app:python_proofs}

The following python script can be used to verify all the theorems and corresponding
proofs used to construct the entropy\dadd{-}conservative and entropy\dadd{-}stable solid wall
boundary conditions proposed herein. The script takes into account the general case
of curvilinear grids.

\lstset{
  language=Python,
  basicstyle=\scriptsize\ttfamily,
  keywordstyle=\color{blue},
  commentstyle=\color{black!60!green},
  morecomment=[l]{!\ },
  numbers=left,
  numbersep=5pt,
  stepnumber=1,
  numberstyle=\tiny\color[rgb]{0.5,0.5,0.5},
  %emph={WallBCInviscid,WallBCAdiabatic},
  %moreemph={WallBC,dWdV,dVdW},
  emphstyle=\color{black!30!red},
}
\begin{lstlisting}
import numpy as np
from sympy import *
init_printing()


# Define few utility routines

def zeros(*shape):
    return np.zeros(shape, dtype=object)

def dot(a, b):
    a = np.array(a)
    b = np.array(b)
    return np.dot(a, b)

def tensordot(a, b, axes=2):
    a = np.array(a)
    b = np.array(b)
    return np.tensordot(a, b, axes)

def matmul(a, b):
    a = np.array(a)
    b = np.array(b)
    return np.dot(a, b)

def outer(a, b):
    a = np.array(a)
    b = np.array(b)
    return np.outer(a, b)


# Viscosity and Conductivity are functions of temperature
Mu = symbols('mu', real=True, positive=True, cls=Function)
Kappa = symbols('kappa', real=True, positive=True, cls=Function)

# All these are constants for the analysis
R = symbols('R', real=True, positive=True)
gamma = symbols('gamma', real=True, positive=True)
cp = R*gamma/(gamma-1)
cv = R/(gamma-1)
rho_inf = symbols('rho_inf', real=True, positive=True)
T_inf = symbols('T_inf', real=True, positive=True)

def pressure(rho, T):
    return rho*R*T

def specific_energy(T):
    return cv*T

def specific_enthalpy(T):
    return cp*T

def specific_entropy(rho, T):
    s = cv*log(T/T_inf) - R*log(rho/rho_inf)
    return s

def kinetic_energy(u):
    return dot(u,u)/2

def EntropyVars(V):
    rho = V[0] # density
    u = V[1:4] # velocity
    T = V[4]   # temperature
    h = specific_enthalpy(T)
    s = specific_entropy(rho, T)
    k = kinetic_energy(u)
    W = zeros(5)
    W[0]   = h/T - k/T - s
    W[1:4] = u/T
    W[4]   = -1/T
    return W

def PrimitiveVars(W):
    T = -1/W[4]
    u = W[1:4]*T
    h = specific_enthalpy(T)
    k = kinetic_energy(u)
    s = h/T - k/T - W[0]
    rho = rho_inf*exp((cv*log(T/T_inf) - s)/R)
    V = zeros(5)
    V[0]   = rho
    V[1:4] = u
    V[4]   = T
    return V

def dWdV(Vs):
    V = zeros(5)
    V[0]   = symbols('rho', real=True, positive=True)
    V[1:4] = symbols('u0:3', real=True)
    V[4]   = symbols('T', real=True, positive=True)
    W = EntropyVars(V)
    mat = zeros(5, 5)
    for a in range(5):
        for b in range(5):
            mat[a,b] = W[a].diff(V[b]).subs(zip(V, Vs))
    return mat

def dVdW(Vs):
    mat = Matrix(dWdV(Vs)).inv()
    return np.array(mat.tolist())

def InviscidFlux(n, V):
    rho = V[0] # density
    u = V[1:4] # velocity
    T = V[4]   # temperature
    p = pressure(rho, T)
    h = specific_enthalpy(T)
    k = kinetic_energy(u)
    H = h + k # total enthalpy
    fI = zeros(5)
    fI[0]   = rho*dot(n, u)
    fI[1:4] = fI[0]*u + p*n
    fI[4]   = fI[0]*H
    return fI

def LogAverage(a, b):
    if a == b: return a
    return (b-a)/log(b/a)

def EntropyConsistentFlux(n, Vl, Vr):
    # Ismail & Roe, doi:10.1016/j.jcp.2009.04.021
    gp1og = (gamma+1)/gamma
    gm1og = (gamma-1)/gamma

    sqrtT_l = sqrt(Vl[4])
    sqrtT_r = sqrt(Vr[4])

    z1_l = 1/sqrtT_l
    z1_r = 1/sqrtT_r
    z1_avg = z1_l + z1_r
    z1_avg_inv = 1/z1_avg
    z1_log = LogAverage(z1_l, z1_r)

    z3_l = sqrtT_l * Vl[0]
    z3_r = sqrtT_r * Vr[0]
    z3_log = LogAverage(z3_l, z3_r)

    rho_hat = z1_avg * z3_log / 2
    p_hat = R * z1_avg_inv * (z3_l + z3_r)
    u_hat = z1_avg_inv * (z1_l * Vl[1:4] + z1_r * Vr[1:4])
    T_hat = 1/rho_hat * (gp1og * z3_log/z1_log + gm1og/R * p_hat)/2

    h_hat = specific_enthalpy(T_hat)
    k_hat = kinetic_energy(u_hat)
    H_hat = h_hat + k_hat # total enthalpy

    fI = zeros(5)
    fI[0]   = rho_hat * dot(n, u_hat)
    fI[1:4] = fI[0] * u_hat +  p_hat * n
    fI[4]   = fI[0] * H_hat
    return fI

def ViscousFlux(n, V, Theta):
    u = V[1:4] # velocity
    T = V[4]   # temperature
    mu, kappa = Mu(T), Kappa(T) # mu and kappa depend on temperature
    Pi = matmul(dVdW(V), Theta) # rotate gradient to primitive vars
    grad_u = Pi[1:4,:] # gradient of velocity
    grad_T = Pi[4,:] # # gradient of temperature
    div_u = np.trace(grad_u) # divergence of velocity
    I = eye(3) # identity tensor
    epsilon = zeros(3,3) # rate-of-strain tensor
    for i in range(3):
        for j in range(3):
            epsilon[i,j] += (grad_u[i,j] + grad_u[j,i])/2
    tau = mu * (2*epsilon - Rational(2,3)*div_u*I) # Cauchy stress tensor
    fV = zeros(5) # normal viscous flux
    fV[0]   = sympify(0)
    fV[1:4] = dot(n, tau)
    fV[4]   = dot(n, dot(u, tau)) + kappa * dot(n, grad_T)
    return fV

# From now on, we use the following conventions:
# * (-) and '_m' refers to internal state at a point within a cell boundary face
# * (w) and '_w' refers to given boundary values, superscript (wall) in the paper
# * (+) and '_p' refers to manufactured state, superscript (B) in the paper

# Unit face normal
nrm = symbols('n0:3', real=True)
nrm = np.array(nrm)
nrm = nrm / sqrt(dot(nrm, nrm))

# Index 'k' identifying the boundary face.
# If k is None, use an arbitrary face normal.
# Otherwise, set k from 1 to 6 to assume a regular hexahedron [0,1]^3
# (unit volume, unit face areas). In the paper we pick k = 1 for the sake
# of simplicity, in here we default to k = None to proof the general case.
k = None
if k == 1: nrm[:] = sympify([-1,0,0])
if k == 2: nrm[:] = sympify([+1,0,0])
if k == 3: nrm[:] = sympify([0,-1,0])
if k == 4: nrm[:] = sympify([0,+1,0])
if k == 5: nrm[:] = sympify([0,0,-1])
if k == 6: nrm[:] = sympify([0,0,+1])

# Wall velocity, satisfies dot(nrm, u_w) == 0, Remark 2.1
u_w = symbols('u0:3^(wall)', real=True)
u_w = np.array(u_w)
u_w = u_w - dot(u_w, nrm) * nrm

# Wall entropy flux function, Equation 20
g = symbols('g', real=True, cls=Function)
t = symbols('t', real=True) # time

# State (-) at a point in the face
V_m = zeros(5)
V_m[0]   = symbols('rho', real=True, positive=True)
V_m[1:4] = symbols('u0:3', real=True)
V_m[4]   = symbols('T', real=True, positive=True)

# Gradient (-) at a point in the face
Theta_m = Matrix(MatrixSymbol('Theta', 5, 3))

# --------------------------------
# --- Proof for Inviscid Terms ---
# --------------------------------

# Manufactured (+) state at a point in the face
V_p = zeros(5)
u_m = V_m[1:4] # velocity at (-)
wn = dot(u_w, nrm) # normal wall velocity, wn == 0, Remark 2.1
un = dot(u_m, nrm) # normal flow velocity
V_p[0] = V_m[0]                          # \
V_p[1] = V_m[1] + 2 * (wn - un) * nrm[0] # |
V_p[2] = V_m[2] + 2 * (wn - un) * nrm[1] # | Equation 31
V_p[3] = V_m[3] + 2 * (wn - un) * nrm[2] # |
V_p[4] = V_m[4]                          # /

W_m = EntropyVars(V_m) # entropy vars at (-)

fI_m  = InviscidFlux(nrm, V_m)
fI_s  = EntropyConsistentFlux(nrm, V_m, V_p)
sat_I = fI_m - fI_s # inviscid penalty term, Equation 28

rho_m, u_m = V_m[0], V_m[1:4] # density and temperature at (-)
Psi_m = R * rho_m * dot(nrm, u_m)
F_m = dot(W_m, fI_m) - Psi_m # entropy flux in LHS of Equation 26

# Inviscid contributions to the RHS of Equation 26
RHS_I = (
    - F_m
    + dot(W_m, sat_I)
)

RHS_I = ratsimp(RHS_I) # simplify expressions
RHS_I = expand(RHS_I)  # trigger cancellations
assert RHS_I == 0      # Entropy-conservative BC! Q.E.D.


# --------------------------------
# --- Proof for Viscous Terms ----
# --------------------------------

# Manufactured (+) state at a point in the face
# Note that this way u^(w) = 1/2*(u^(-) + u^(+))
V_p = zeros(5)
V_p[0] = +V_m[0]              # \
V_p[1] = -V_m[1] + 2 * u_w[0] # |
V_p[2] = -V_m[2] + 2 * u_w[1] # | Equation 32
V_p[3] = -V_m[3] + 2 * u_w[2] # |
V_p[4] = +V_m[4]              # /

# Manufactured (+) gradient at a point in the face
Pi_p = zeros(5, 3)
Pi_m = matmul(dVdW(V_m), Theta_m) # Equation 36
Pi_p[0,:] = -Pi_m[0,:]            # \
Pi_p[1,:] = +Pi_m[1,:]            # |
Pi_p[2,:] = +Pi_m[2,:]            # | Equation 37
Pi_p[3,:] = +Pi_m[3,:]            # |
Pi_p[4,:] = -Pi_m[4,:]            # /
Theta_p = matmul(dWdV(V_p), Pi_p) # Equation 38

W_m = EntropyVars(V_m) # entropy vars at (-)
W_p = EntropyVars(V_p) # entropy vars at (+)

fV_m = ViscousFlux(nrm, V_m, Theta_m)
fV_p = ViscousFlux(nrm, V_p, Theta_p) # Equation 33
sat_V = -(fV_m-fV_p)/2 # viscous penalty term, Equation 29

# Source heat flux corresponding to g(t)
src_L = zeros(5)
T_m = V_m[4] # temperature at (-)
src_L[4] = -T_m * g(t) # Equation 35

# Volume viscous flux to contract with gradient Equation 23b
fV = zeros(5,3)
fV[:,0] = ViscousFlux([1,0,0], V_m, Theta_m)
fV[:,1] = ViscousFlux([0,1,0], V_m, Theta_m)
fV[:,2] = ViscousFlux([0,0,1], V_m, Theta_m)
# Gradient penalty term
sat_Theta = -Rational(1,2) * outer(W_m - W_p, nrm) # Equation 30

# Viscous contributions to the RHS of Equation 26
RHS_V = (
    + dot(W_m, fV_m)
    + dot(W_m, sat_V)
    + dot(W_m, src_L)
    + tensordot(fV, sat_Theta)
)

RHS_V = ratsimp(RHS_V) # simplify expressions
RHS_V = expand(RHS_V)  # trigger cancellations
assert RHS_V == g(t)   # Entropy-conservative BC! Q.E.D.


# --------------------------------
# --- Proof for IP dissipation ---
# --------------------------------

# This constant is the IP strength factor
beta = symbols('beta', real=True, positive=True)

# Proof using viscous flux evaluations and normal entropy jumps,
# this is the practical way to implement the IP dissipation.
# This proof holds for any normal vector and wall velocity. Note
# that in the paper, and for the sake of simplicity, we express the
# IP dissipation in terms of the # viscous C_jj, j=1,2,3 and we
# assume normal vectors aligned with the Cartesian coordinate directions.

dWn = outer(W_m - W_p, nrm)
fm = ViscousFlux(nrm, V_m, dWn)
fp = ViscousFlux(nrm, V_p, dWn)
M = -beta * (fm + fp)/2 # Equation 40
RHS_IP = dot(W_m, M) # Equation 41

RHS_IP = expand(RHS_IP)
RHS_IP = simplify(RHS_IP)

# Now we check that the IP term is non-positive
u_m = V_m[1:4]
T_m = V_m[4]
du = u_m - u_w
N  = outer(nrm, nrm) # rank 1, symmetric, semi-PD matrix
RHS_IP_explicit = (
    -2*beta*Mu(T_m)/(3*T_m) *  ( # this factor is negative
        + dot(du, matmul(N, du)) # this term is non-negative
        + 3*dot(du,du)*dot(nrm,nrm) # this term is non-negative
    ) # then this is non-positive and thus dissipative
)
RHS_IP_explicit = expand(RHS_IP_explicit)
RHS_IP_explicit = simplify(RHS_IP_explicit)
assert expand(RHS_IP - RHS_IP_explicit) == 0 # Q.E.D

# Proof using the C_jj, j=1,2,3 viscous matrices,
# this is the usual form we use in proofs.
# This proof DOES NOT HOLD for any normal vector and wall velocity,
# it assumes that the normal vector is aligned with the Cartesian
# coordinate directions and dot(nrm, u_w) == 0.

def C_11(V):
    rho,u1,u2,u3,T = V
    mu, kappa = Mu(T), Kappa(T)
    C = zeros(5,5)
    C[1,1] = Rational(4,3)*T*mu
    C[1,4] = Rational(4,3)*T*mu*u1
    C[2,2] = T*mu
    C[2,4] = T*mu*u2
    C[3,3] = T*mu
    C[3,4] = T*mu*u3
    C[4,1] = C[1,4]
    C[4,2] = C[2,4]
    C[4,3] = C[3,4]
    C[4,4] = T**2*kappa+Rational(1,3)*T*mu*(4*u1**2+3*u2**2+3*u3**2)
    return C

def C_22(V):
    rho,u1,u2,u3,T = V
    mu, kappa = Mu(T), Kappa(T)
    C = np.zeros([5,5], dtype=object)
    C[1,1] = T*mu
    C[1,4] = T*mu*u1
    C[2,2] = Rational(4,3)*T*mu
    C[2,4] = Rational(4,3)*T*mu*u2
    C[3,3] = T*mu
    C[3,4] = T*mu*u3
    C[4,1] = C[1,4]
    C[4,2] = C[2,4]
    C[4,3] = C[3,4]
    C[4,4] = T**2*kappa+Rational(1,3)*T*mu*(3*u1**2+4*u2**2+3*u3**2)
    return C

def C_33(V):
    rho,u1,u2,u3,T = V
    mu, kappa = Mu(T), Kappa(T)
    C = zeros(5, 5)
    C[1,1] = T*mu
    C[1,4] = T*mu*u1
    C[2,2] = T*mu
    C[2,4] = T*mu*u2
    C[3,3] = Rational(4,3)*T*mu
    C[3,4] = Rational(4,3)*T*mu*u3
    C[4,1] = C[1,4]
    C[4,2] = C[2,4]
    C[4,3] = C[3,4]
    C[4,4] = T**2*kappa+Rational(1,3)*T*mu*(3*u1**2+3*u2**2+4*u3**2)
    return C

if k is not None:

    if k in (1, 2):
        C_m = C_11(V_m)
        C_p = C_11(V_p)
    if k in (3, 4):
        C_m = C_22(V_m)
        C_p = C_22(V_p)
    if k in (5, 6):
        C_m = C_33(V_m)
        C_p = C_33(V_p)

    L = -beta * (C_m + C_p)/2 # Equation 39
    M = dot(L, W_m - W_p) # Equation 40
    RHS_IP_v2 = dot(W_m, M) # Equation 41
    RHS_IP_v2 = expand(RHS_IP_v2)
    RHS_IP_v2 = simplify(RHS_IP_v2)

    assert RHS_IP == RHS_IP_v2
\end{lstlisting}

\section{FORTRAN code for the implementation of the boundary conditions on 
  curvilinear grids}\label{app:fortran_code}

In this appendix we provide a simple but yet general FORTRAN implementation of the 
proposed entropy stable wall boundary conditions on curvilinear grids. The following
piece of code receives as input the primitive variables, $V$, and outputs the 
primitive variables of the ghost state. It is supposed to be called for each 
collocated point lying on the wall boundary face.

\lstset{
  language=[95]Fortran,
  basicstyle=\scriptsize\ttfamily,
  keywordstyle=\color{blue},
  commentstyle=\color{black!60!green},
  morecomment=[l]{!\ },
  numbers=left,
  numbersep=5pt,
  stepnumber=1,
  numberstyle=\tiny\color[rgb]{0.5,0.5,0.5},
  emph={WallBCInviscid,WallBCAdiabatic},
  %moreemph={WallBC,dWdV,dVdW},
  emphstyle=\color{black!30!red},
}
\begin{lstlisting}
module WallBC

  integer, parameter :: d     = 3               ! dimension in {2,3}
  integer, parameter :: c     = d+2             ! components
  integer, parameter :: I1    = 1               ! index of density
  integer, parameter :: I2(d) = (/(i,i=1,d)/)+1 ! index of velocity
  integer, parameter :: I3    = c               ! index of temperature

  integer, parameter :: dp = selected_real_kind(15, 307)

contains

  pure subroutine WallBCInviscid(Snrm,u_wall,V_m,V_p)
    real(kind=dp), intent(in)  :: Snrm(d)   ! normal vector
    real(kind=dp), intent(in)  :: u_wall(d) ! wall velocity vector
    real(kind=dp), intent(in)  :: V_m(c)    ! primitive vars
    real(kind=dp), intent(out) :: V_p(c)    ! primitive vars
    real(kind=dp)  :: nrm(d) ! unit normal
    real(kind=dp)  :: wn     ! normal velocity of wall
    real(kind=dp)  :: un     ! normal velocity of fluid

    nrm = Snrm/sqrt(dot_product(Snrm,Snrm))
    wn  = dot_product(nrm,u_wall)
    un  = dot_product(nrm,V_m(I2))

    V_p(I1) = V_m(I1)
    V_p(I2) = V_m(I2) + 2 * (wn - un) * nrm
    V_p(I3) = V_m(I3)

  end subroutine WallBCInviscid

  pure subroutine WallBCAdiabatic(u_wall,V_m,dWdX_m,V_p,dWdX_p)
    real(kind=dp), intent(in)  :: u_wall(d)   ! wall velocity vector
    real(kind=dp), intent(in)  :: V_m(c)      ! primitive vars
    real(kind=dp), intent(in)  :: dWdX_m(c,d) ! gradient of entropy vars
    real(kind=dp), intent(out) :: V_p(c)      ! primitive vars
    real(kind=dp), intent(out) :: dWdX_p(c,d) ! gradient of entropy vars
    real(kind=dp)  :: dVdX_m(c,d) ! gradient of primitive vars
    real(kind=dp)  :: dVdX_p(c,d) ! gradient of primitive vars

    V_p(I1) = +V_m(I1)
    V_p(I2) = -V_m(I2) + 2 * u_wall
    V_p(I3) = +V_m(I3)

    dVdX_m = matmul(dVdW(V_m),dWdX_m)
    dVdX_p(I1,:) = -dVdX_m(I1,:)
    dVdX_p(I2,:) = +dVdX_m(I2,:)
    dVdX_p(I3,:) = -dVdX_m(I3,:)
    dWdX_p = matmul(dWdV(V_p),dVdX_p)

  end subroutine WallBCAdiabatic

  pure function dWdV(V)
    real(kind=dp), intent(in) :: V(c)
    real(kind=dp)             :: dWdV(c,c)
    dWdV = ... ! Fill-in with proper values: exercise left to the reader
  end function dWdV

  pure function dVdW(V)
    real(kind=dp), intent(in) :: V(c)
    real(kind=dp)             :: dVdW(c,c)
    dVdW = ... ! Fill-in with proper values: exercise left to the reader
  end function dVdW

end module WallBC
\end{lstlisting}

\end{appendix}

\end{document}